\documentclass[12pt]{article}               
\usepackage{geometry}
\usepackage{multirow}
\usepackage{amsmath}
\usepackage{amssymb}
\usepackage[usenames]{color}
\usepackage{graphicx}          
 \usepackage{amsthm}                              
\usepackage[dvips]{epsfig}    
\newtheorem{definition}{Definition}
\newtheorem{theorem}{Theorem}

\newtheorem{remark}{Remark}

\newtheorem{lemma}{Lemma}

\begin{document}


\centerline{ \large \bf On a class of generating vector fields}

\centerline{ \large \bf for the extremum seeking problem:}

\centerline{ \large \bf Lie
bracket approximation and stability properties$^{*}$}
\quad\\

\centerline{ \large Victoria Grushkovskaya$^{1,3}$, Alexander Zuyev$^{2,3}$, Christian Ebenbauer$^1$} 

\let\thefootnote\relax\footnote{$^{*}$This work was supported in part by the Alexander von Humboldt Foundation and the Deutsche Forschungsgemeinschaft (EB 425/4-1). Corresponding author V.~Grushkovskaya. }
\footnotetext[1]{$^1$Institute for Systems Theory and Automatic Control, University  of Stuttgart, 70569 Stuttgart, Germany, \{grushkovskaya, ce\}@ist.uni-stuttgart.de}
\footnotetext[2]{$^2$Max Planck Institute for Dynamics of Complex Technical Systems, 39106 Magdeburg, Germany, zuyev@mpi-magdeburg.mpg.de}
\footnotetext[3]{$^3$Institute of Applied Mathematics and Mechanics, National Academy of Sciences of Ukraine,  84100 Sloviansk, Ukraine}

{\centering \small Keywords:                           
Extremum seeking;  asymptotic stability; approximate gradient flows; Lie bracket approximation; vibrational stabilization.            
}                         

\begin{abstract}                          
In this paper, we describe a broad class of control functions for  extremum seeking problems. We show that it
unifies and generalizes existing extremum seeking  strategies
which are based on Lie bracket approximations, and allows to design
new controls with favorable properties in extremum seeking and
vibrational stabilization tasks. The second result of this paper is a novel approach for studying the asymptotic behavior of extremum seeking systems. It provides a constructive procedure for defining frequencies of control functions to ensure the
practical asymptotic and exponential stability.
In contrast to many known results, we also prove asymptotic and exponential stability in the sense of Lyapunov for the proposed class of extremum seeking systems under
appropriate assumptions on the vector fields.
\end{abstract}


\section{Introduction}
In many control applications, the goal is to operate a system in some optimal fashion. Often, however, the optimal operating point is unknown or may even change over  time so that it cannot be determined a priori. Extremum seeking control is a control
methodology to solve such problems of stabilizing and tracking an a priori unknown optimal operating point. Typically, it is model-free and minimizes or maximizes the steady-state map of a system. The steady-state map maps constant control input values to
the steady-state output values.
It is a well-defined map under appropriate assumptions on the system.
There exist many ways to design the extremum seeking strategies.
A classical perturbation-based approach is to use the controls consisting  of time-periodic oscillating inputs (often called dither, excitation, perturbation  or learning signal) and state-dependent  vector fields in order to gather information about the
unknown steady-state map.
Based on the perturbed input and the perturbed output response, typically the gradient
or other descent directions of the steady-state map are approximated or estimated by appropriate signal processing or filtering methods, see, e.g. \cite{Durr17,Due-Sta-Ebe-Joh-12,King12,Guay15,Guay03,Har16,Krst03,Kr00,Nes10,review}.
Hereby, the shape of control functions plays an important role
since it influences the  speed of convergence and may be subject to input constraints.
In the literature, different types of  excitation signals have been analyzed,
see, e.g. ~\cite{Chi07, Nes09, Sch14, Sch16,TNM08}.\\
In this paper, we propose a novel class of vector fields  for extremum seeking controls based on Lie bracket approximation techniques \cite{Due-Sta-Ebe-Joh-12,Durr17}.
The first contribution of this paper is a formula describing a whole class of vector fields for an extremum seeking system which allows to approximate a gradient flow in various ways. The formula unifies and generalizes previously known controls presented
in \cite{Due-Sta-Ebe-Joh-12,Sch14,Sch16,SD17} and allows to generate new extremum seeking strategies with desirable  properties.
In particular, we demonstrate benefits of this  formula by designing a control
which has bounded update rates \textit{and} vanishing amplitudes at the same time.\\
Moreover, the second contribution  is a rigorous proof of the asymptotic and exponential stability \textit{in the sense of Lyapunov},
under appropriate assumptions on the considered class of generating vector fields.
This is in contrast to many results in the literature, where typically
\textit{practical} stability results are established. The proof also extends the techniques developed in \cite{ZuSIAM,ZGB,GrZuNA,GrZuIEEE}
{to a wide class of cost functions and to systems whose vector fields are non-differentiable at the origin}.  An advantage of these techniques is the possibility to estimate the decay rate of solutions of the extremum
seeking systems. \\
Finally, we demonstrate that the proposed formula is not only of use in
extremum seeking but also in vibrational stabilization problems \cite{SK13,ME}.\\
The paper is organized as follows. Section~\ref{Prel} contains
some preliminary results on  extremum seeking  based on Lie bracket approximations.
In Section~\ref{Main}, we present a novel formula to approximate the gradient flows and establish various asymptotic stability conditions. In  Section~\ref{appl}, we illustrate several extremum seeking strategies by using numerical simulations, and
discuss the application of the obtained results to the vibration stabilization problem.  Appendix~\ref{proofs} contains  auxiliary lemmas and proofs.
\section{Preliminaries}~\label{Prel}
\subsection{Notations}
Throughout the text, $\mathbb R^+$  denotes the set of all non-negative real numbers,
 $B_\delta(x^*)$ is the $\delta$-neighborhood of $x^*{\in} \mathbb R^n$,  $\overline{B_\delta(x^*)}$ is its closure. For  $h{\in} C^1(\mathbb R^n;\mathbb R)$, $\xi{\in}\mathbb R^n$, 
we define the column  $\nabla h(\xi):=\frac{\partial h(x)}{\partial x}^T\Big|_{x=\xi}$. For a function $f:\mathbb R\to\mathbb R$,  $f(z)=O(z)$ as $z\to z^0$ means that there is a $c>0$ such that $|f(z)|\le c|z|$ in some neighborhood of $z^0$. { For  $f,g:\mathbb R^n\to\mathbb R^n $, $x\in\mathbb R^n$, we denote the Lie derivative as
 {$L_gf(x)=\lim\limits_{s\to0}\tfrac{f(x+sg(x))-f(x)}{s}$}, and  $[f,g](x)=L_fg(x)-L_gf(x)$ is the Lie bracket. For $m,n\in\mathbb Z$, the notation $i=\overline{m,n}$ means that $i=m,m+1,\dots,n$.}
 {For $a,b\in\mathbb R^n$, we denote their open convex hull as ${\rm co}\{a,b\}=\{\lambda a+ (1-\lambda)b\,|\,\lambda\in (0,1)\}$.}
\subsection{Lie bracket approximations}
Consider a control-affine system
\begin{equation}\label{aff}
  \dot x=f_0(x)+\sum_{j=1}^\ell f_j(x){\sqrt\omega} u_j({\omega} t),
\end{equation}
where $x{=}(x_1,\dots,x_n)^T{\in}\mathbb R^n$, $x(t_0){=}x^0{\in}\mathbb R^n$ (without loss of generality, we assume $t_0=0$), $\omega{>}0$,
$f_j{:\mathbb R^n\to\mathbb R^n}$, $j{=}\overline{1,\ell}$.
{Assume that:\\
A0 $u_j(t)$ are continuous  $T$-periodic functions, $\int_0^Tu_j(\tau)d\tau{=}0$,
$ \int_0^T\int_0^\theta  u_{i}(\theta) u_{j}(\tau)d\tau ds{{=}}\beta_{i,j} T$, $T{>}0$, $\beta_{i,j}\in\mathbb R$, $i,j=\overline{1,\ell}$}.\\
It can be shown that the trajectories of~\eqref{aff} approximate trajectories of
the following \textit{Lie bracket system}:
\begin{equation}\label{affLie}
    \dot {\bar x}{=}f_0(\bar x){+}\sum\limits_{i{<}j}{\beta_{j,i}}[f_i,f_j](\bar x),\quad\bar x(0)=x^0.
  \end{equation}
	The stability properties of systems~\eqref{aff} and \eqref{affLie} are related as follows.
  \begin{lemma}[\cite{Due-Sta-Ebe-Joh-12}]~\label{dthm}
{Let $f_0,f_i\in C^2(\mathbb R^n;\mathbb R^n)$, and $u_i$ satisfy A0,  $i=\overline{1,\ell}$.}  If a compact set $ S\subset\mathbb R^n$ is locally (globally) uniformly asymptotically stable for~\eqref{affLie} then it is locally (semi-globally) practically
uniformly asymptotically stable for~\eqref{aff}.
  \end{lemma}	
Below we recall the notion of practical stability.
	 \begin{definition}
A compact set   $S\subset\mathbb R^n$ is said to be {locally practically uniformly asymptotically stable} for~\eqref{aff} if:\\
-- it is {practically uniformly stable}, i.e. for every $\varepsilon{>}0$ there exist $\delta{>}0$ and $\omega_0{>}0$ such that, for all $t_0{\ge}0$ and $\omega{>}\omega_0$,  if $x^0{\in} B_{\delta}(S)$ then the corresponding solution of~\eqref{aff}
satisfies $x(t){\in}  B_{\varepsilon}(S)$ for all $t{\ge} t_0$;\\
-- $\hat\delta$-{practically uniformly attractive} with some $\hat\delta{>}0$, i.e. for every $\varepsilon{>}0$ there exist  $t_1{\ge}0$ and $\omega_0{>}0$ such that, for all $t_0{\ge}0$ and $\omega{>}\omega_0$,  if $x^0{\in} B_{\hat\delta}(S)$ then the
corresponding solution of~\eqref{aff} satisfies $x(t){\in} B_{\varepsilon}(S)$ for all $t{\ge} t_0{+}t_1$.\\
 If the attractivity property holds for every $\hat\delta{>}0$, then the set $S$ is called  {semi-globally practically uniformly asymptotically stable} for~\eqref{aff}.
\end{definition}
\subsection{Extremum seeking problem}
In this paper, we address a { class of extremum seeking problems related to the unconstrained minimization of a cost function $J$. 
 We assume that $J\in C{^2}(\mathbb R^n;\mathbb R)$ is unknown (as an analytic expression) but can be evaluated (measured)  at each $x \in \mathbb R^n.$ The goal is to construct a control system of the form}
$\displaystyle
 \dot x = u(t,J(x))
$
such that the (local) minima of $J$ have some desired stability properties for this system.
In this setup, a static map $J$ corresponds to the steady-state map of a system.
However, the extremum seeking based on Lie bracket approximations can be applied to much more general scenarios, including dynamic maps (dynamical systems), constrained optimization problems, distributed and multi-agent extremum seeking, stabilization,
synchronization and consensus problems as well as problems on manifolds, etc. 
The results obtained in this paper can be  applied to such more general problems but are not discussed here for the sake of simplicity. \\
The underlying idea of the extremum seeking based on the Lie bracket approximations is as follows.
Suppose that $n=1$, i.e. $x\in\mathbb R$, and consider the system
\begin{equation}\label{int1}
\dot x=J(x)\sqrt\omega\cos(\omega t)+\sqrt\omega\sin(\omega t).
\end{equation}
It can be seen that the Lie bracket system for~\eqref{int1} approximates the gradient flow of $J$:
\begin{equation}\label{int1_Lie}
\dot{\bar x}=[J(\bar x),1]=-\tfrac{1}{2}\nabla J(\bar x).
\end{equation}
Thus, the trajectories of system~\eqref{int1} approximate trajectories of the gradient flow of $J$ and
they converge, for example, if $J$ is convex and has minima, into an arbitrary small neighborhood of the set of minima of $J$,
for sufficiently large $\omega$.  For $n>1$, the gradient flow can be approximated in a similar way,
see \cite{Due-Sta-Ebe-Joh-12} for details.
\newpage
\section{Main results}~\label{Main}
\subsection{Vector fields for approximating gradient flows}~\label{sec_grad}
Observe that there are many ways to  define the vector fields of system~\eqref{aff}
such that 
the corresponding Lie bracket system has the form~\eqref{int1_Lie}. For example, consider the system
\begin{align}
\dot x=\tfrac{1}{2}e^{J(x)}\sqrt\omega\cos(\omega t)+e^{-J(x)}\sqrt\omega\sin(\omega t).
\end{align}
Computing $\big[\tfrac{1}{2}e^{J(x)},e^{-J(x)}\big]$ yields $-\nabla J(x)$ and, hence, the associated Lie bracket system is again of the form~\eqref{int1_Lie}.\\
The main idea and  the first main result of this paper is the description of a class of vector fields for system~\eqref{aff} such that the corresponding Lie bracket system~\eqref{affLie} represents a gradient-like flow of $J$.
Consider first the system
\begin{equation}\label{int}
\dot x=F_{1}(J(x)){\sqrt\omega }u_{1}(\omega t)+F_{2}(J(x)){\sqrt\omega} u_{2}(\omega t).
\end{equation}
 We begin with the one-dimensional case $x{\in}\mathbb R$ to simplify the presentation, and the multi-dimensional case will be considered later as an extension.
\begin{theorem}~\label{thm_appr}
Let the functions $F_{1},F_{2}\in C{^1}(\mathbb R; \mathbb R)$  satisfy
\begin{equation}\label{class}
F_2(z)={-}F_1(z)\int{\frac{F_0(z)}{F_1(z)^2}}dz,
\end{equation}
with some $F_{0}:\mathbb R{\to }\mathbb R$, 
and let $u_{s}: \mathbb R^+{\to}\mathbb R$ satisfy $A0$.
 Then the Lie bracket system for~\eqref{int} has the form
  \begin{equation}\label{int_Lie}
   \dot {\bar x}=-\beta_{2,1} \nabla J(\bar x) F_0(J(\bar x)).
 \end{equation}
\end{theorem}
\textbf{PROOF.}
Consider the differential equation
\begin{equation}\label{pfaff}
F_2(z)\frac{dF_1(z)}{dz}-  F_1(z)\frac{dF_2(z)}{dz}=F_0(z),\quad z\in\mathbb R,
\end{equation}
{and observe that it represents a linear ordinary differential equation with respect to $F_2(z)$ (or $F_1(z)$), whose solutions are given by~\eqref{class}.}
Then, by computing the Lie bracket system for~\eqref{int}, we obtain~\eqref{int_Lie}:
\begin{align*}
& \dot {\bar x}{=}\tfrac{1}{T}[F_1(J(\bar x)),F_2(J(\bar x))]\int_0^T\int_0^\theta u_2(\theta) u_1(\tau)d\tau d\theta\\
   &{=}\beta_{2,1}\left(F_1(J(\bar x))\frac{d F_2(J(\bar x))}{d J}{-}F_2(J(\bar x))\frac{d F_1(J(\bar x))}{dJ}\right)\nabla J(\bar x) \\
   &{=}{-}\beta_{2,1}\nabla J(\bar x)F_0(J(\bar x)). \hspace{14em}\square
\end{align*}
Formula~\eqref{class} describes the whole class of solutions of~\eqref{pfaff}, and this is the unique way to obtain the Lie bracket system~\eqref{int_Lie} if the original system has the form~\eqref{int}. However, there are many ways to obtain the Lie bracket system of the type~\eqref{int_Lie} if the original system has the form $\dot x=u$, e.g., by expressing the gradient-like dynamics as a linear combination of certain Lie brackets.
\begin{remark}
  In formula~\eqref{class}, we assume that $F_1(z)\ne0$ except for at most a countable set of isolated zeros $Z^*=\{z_k^*\}$.
  We treat the function $\Psi_1(z):=\int{{\tfrac{F_0(z)}{F_1(z)^2} }dz }$ as an antiderivative of $\displaystyle\tfrac{F_0(z)}{F_1(z)^2}$ defined on the open set ${\mathbb R}\setminus Z^*$, so that~\eqref{class} holds as an identity with continuous functions in a neighborhood of each point $z\notin Z^*$.
  As the functions $F_1$ and $F_2$ are assumed to be globally continuous, formula~\eqref{class} is treated
  in the sense that $F_2(z_k^*)=-\lim_{z\to z_k^*}F_1(z)\Psi_1(z)$ at each $z_k^*\in Z$.
\end{remark}
Formula~\eqref{class} can also be used to approximate gradient-like flows of multivariable cost functions. Consider the system\\
{\begin{equation}\label{int_n}
\dot x=\sum_{i=1}^n\Big(F_{1i}(J(x))u_{1i}( t)+F_{2i}(J(x)) u_{2i}( t)\Big)e_i,\,x\in\mathbb R^n,
\end{equation}
where  $J{\in}C^2(\mathbb R^n;\mathbb R)$,  $e_i$ denotes the $i$-th unit vector in $\mathbb R^n$.
\begin{theorem}~\label{thm_pr}
Suppose that each pair $F_{1i},F_{2i}\in C^1(\mathbb R;\mathbb R)$ satisfies  relation~\eqref{class} with some $F_{0i}:\mathbb R\to\mathbb R$. Define $u_{si}( t):=\sqrt\omega\tilde u_{si}(\omega t)$, $s=1,2$, $i=\overline{1,n}$, $\omega>0$, where the
functions $\tilde u_{si}$ {satisfy A0.}
 Then the  Lie bracket system for~\eqref{int_n} has the form
\begin{align}
\label{gradflow}
  { \dot{\bar x}}=-\sum_{i=1}^n \beta_{2i,1i}\frac{\partial J(\bar x)}{\partial\bar x_i}F_{0i}(J(\bar x)) e_i.
\end{align}
Moreover, if  $F_{1i},F_{2i}{\in C^2(\mathbb R;\mathbb R)}$ 
and a compact set $ S{\subset}\mathbb R^n$ is locally (globally) uniformly asymptotically stable for~\eqref{gradflow}, then $S$ is locally (semi-globally) practically uniformly asymptotically stable for~\eqref{int_n}.
\end{theorem}}
The proof follows directly from Lemma~\ref{dthm} and Theorem~\ref{thm_appr}.
\subsection{{Examples}}
Formula~\eqref{class}  describes a whole class of vector fields $F_1$, $F_2$ such that the trajectories of system~\eqref{int} approximate trajectories of the gradient-like system~\eqref{int_Lie}.
Moreover, formula~\eqref{class} unifies and generalizes some known results which are discussed in the following. In particular, assume $F_0=1$.\\
In \cite{Due-Sta-Ebe-Joh-12}, the case
$$
 F_1(z)=z,~F_2(z)=1
$$
was considered, which corresponds to system \eqref{int1}.
The paper \cite{Sch14} introduced the functions
$$
 F_1(z)=\mathrm{sin}(z),~F_2(z)=\mathrm{cos}(z),
$$
 which possess a priori known bounds  (i.e. one has bounded update rates) for a fixed $\omega$, due to the property $-1 \le F_s(z) \le 1$, $s=1,2$.\\
The paper~\cite{Sch14b} presented a class of control functions vanishing at the origin. In particular, the control $u=\sqrt{\omega}(\alpha\|x\|^r\cos(\omega t){-}\tfrac{k}{1-r}\|x\|^{2-r}\sin(\omega t))$, $r{\in}[0,1)$, $\alpha,k{>}0$ was proposed, which corresponds to
$J(x){=}\|x\|^{2m}$ with $m{>}0$,
$$
F_1(z){=}\alpha z^{\tfrac{r}{2m}}, F_2(z){=}-\frac{k}{1-r} z^{\tfrac{2-r}{2m}}, F_0(z){=}\frac{1}{m}z^{\tfrac{1-m}{m}},
$$
so that the above control can also be described by formula~\eqref{class}. In~\cite{Sch14b}, practical asymptotic stability conditions for   control-affine system~\eqref{affLie} with vector fields $f_i\in C^2(\mathbb R^n\setminus\{0\};\mathbb R^n)$ were
presented  under the assumption $[f_i,f_j]\in C^2(\mathbb R^n;\mathbb R^n)$.\\
Another case with vanishing at the origin vector fields was studied in \cite{SD17}, namely,
$$
 F_1(z)=\sqrt{z}\mathrm{sin}(\mathrm{ln}(z)),~F_2(z)=\sqrt{z}\mathrm{cos}(\mathrm{ln}(z))
$$
were considered for $z\ge0$.
\\
Besides unifying these known results, formula~\eqref{class} allows also to construct novel
controls with desirable  properties. In particular, it is possible to
combine the advantage of having bounded update rates and vanishing perturbation amplitudes,  e.g.,
this can be achieved with $F_0(z)=1$ and
$$
F_1(z)=\sqrt{\phi_1(z)}\sin(\phi_2(z)),\,F_2(z)=\sqrt{\phi_1(z)}\cos(\phi_2(z)), $$
$$\phi_1(z)=  {\tfrac{1-e^{-z}}{1+e^{z}}},\,\phi_2(z)=e^{z}+2\ln(e^{z}-1),
$$
for $z>0$, $F_1(0)=F_2(0)=0$. For $z\ge0$, one has $|F_s(z)|\le \sqrt{3-2\sqrt2}$, $s=1,2$.
\\
The controls which  tend to zero whenever $z$ approaches the origin  are useful when the minimal value of $J(x^*)$ is known a priori (but not the extremum point $x^*$ itself). For example, such situations arise
in the distance minimization, consensus or synchronization problems and, as we will see in Section~4.2, in vibrational
stabilization problems where $J$ plays the role of a Lyapunov function. Note that the solutions $F_1,F_2$ to the differential equation~\eqref{pfaff} are not necessary of class $C^2$, as  it was illustrated by the above examples. In the next section, we will relax the regularity assumption on $F_{1i},F_{2i}$ and propose new asymptotic stability conditions for system~\eqref{int_n}.\\
\\
\subsection{Stability conditions}
In this section, we establish the second main result of the paper.
{Namely, we  present  conditions for \textit{asymptotic stability in the sense of Lyapunov}, which are in contrast to many existing results in extremum seeking literature stating only the \textit{practical asymptotic stability}. We
  present a novel approach for studying the stability properties of extremum seeking system~\eqref{int_n} with a broad family of vector fields described by~\eqref{class}. 
 We will refer to the following assumptions in a domain $D\subseteq\mathbb R^n$.\\
{
 {A1} \textit{There exists an $x^*{\in} D$  such that $\nabla J(x^*){=}0$, $\nabla J(x){\ne} 0$ for all $x{\in }D{\setminus}\{x^*\}$; $J(x^*){=}J^*{\in}\mathbb R$, $J(x){>}J(x^*)$ for all $x{\in} D{\setminus}\{x^*\}$.}\\
{A2} \textit{ There exist constants $\gamma_1,\gamma_2,\kappa_1,\kappa_2,\mu$, and $m_1\ge 1$, such that, for all {$x\in D$},}
     $$ \begin{aligned}
     \gamma_1\|x{-}x^*\|^{2m_1} {\le} &\tilde J(x){\le} \gamma_2\|x{-}x^* \|^{2m_1},\quad\tilde J(x) = J(x)-J^*,\\
      \kappa_1 \tilde J(x)^{2{-}\frac{1}{m_1}}{\le}& \|\nabla J(x)\|^2{\le}\kappa_2 \tilde J(x)^{2{-}\frac{1}{m_1}},\\
    \left\|\frac{\partial^2 J(x)}{\partial x^2}\right\|{\le}&\mu \tilde J(x)^{1{-}\frac{1}{m_1}}.
    \end{aligned}
    $$}
{A3} \textit{ {The functions $F_{si}(J(\cdot))\in C^2(D\setminus\{x^*\};\mathbb R)$, $D\subseteq\mathbb R^n$; the functions $L_{F_{pj}}F_{si}(J(\cdot))$, $L_{F_{ql}}L_{F_{pj}}F_{si}(J(\cdot))\in C(D;\mathbb R)$, for all   $s,p,q\in\{1,2\}$,
$i,j,l=\overline{1,n}$.}\\
{ A4} \textit{The functions $F_{si}(J(x))$ are Lipschitz continuous on each compact  $\chi\subset D$, and
\begin{align*}
&\alpha_1 \tilde J^{m_2}(x)\le F_{0i}(\tilde J(x))\le \alpha_2 \tilde J^{m_2}(x),\\
&|F_{si}(\tilde J(x))|\le  M \tilde J^{m_3}(x),\\
&\|L_{F_{ql}}L_{F_{pj}}F_{si}(\tilde J(x))\|\le H \tilde J^{m_4}(x),\text{ for all }x\in D,
\end{align*}
for all $s,p,q{=}\overline{1,2}$, $i,j,l{=}\overline{1,n}$, with $m_2{\ge} \tfrac{1}{m_1}-1$,  $m_3{=}\tfrac{1}{2}(m_2+1)$, $m_4=\tfrac{3}{2}(1+m_2)-\tfrac{1}{m_1}$, and some $\alpha_1,\alpha_2,M>0$, $H\ge0$.  
}}\\
{Assumption A1 requires that the cost function $J$ has an isolated local minimum, and the attained minimal value of $J$ at $x^*$ is $J^*$.
  Assumption A2 is obtained from the requirement  that the cost function $J$ locally behaves as a power function.
Assumption A3 requires that the first and the second Lie derivatives of $F_{si}(J(\cdot))$ be continuous, even if $F_{si}(J(\cdot))$ itself are not continuously differentiable in $D$. {Similar assumptions were exploited, e.g.,
in~\cite{Sch14b,SD17} for a certain class of controls}. Finally,  assumption A4 requires that the controls vanish at the extremum point. As it will be shown in Theorem~\ref{thm_class}, if the minimal value of the cost function is known for the functions in~\eqref{int_n}, then   the
asymptotic stability in the sense of Lyapunov can be ensured. A4 is not needed for the practical asymptotic stability. Although the choice of $m_3,m_4$ may seems artificial, it  naturally holds if all $|F_{si}(J(x))|$ are bounded with the same power of $J$. \\
We will use the following trigonometric inputs in~\eqref{int_n} (however,  some other inputs are possible, see, e.g.,~\cite{TNM08}):
 \begin{equation}\label{cont_eps}
\begin{aligned}
& u_{1i}(t){=}u_{1i}^\varepsilon(t){=}2\sqrt{{\pi k_i}{\varepsilon^{-1}}}\cos\big({2\pi k_i t}{\varepsilon^{-1}}\big),\\
& u_{2i}(t){=}u_{2i}^\varepsilon(t){=}2\sqrt{{\pi k_i}{\varepsilon^{-1}}}\sin\big({2\pi k_i t}{\varepsilon^{-1}}\big),
\end{aligned}
\end{equation}
 where $k_i\in\mathbb N$, $k_i\ne k_j$ for all $i\ne j$, and $\varepsilon$ is a positive parameter. Note that such inputs satisfy A0 with $\omega=\varepsilon^{-1}$, $T=\varepsilon$. We underline the dependence of controls on $\varepsilon$ by using the
 superscript $u_{si}^\varepsilon$ in~\eqref{cont_eps}.
\\
{The next result states conditions for the \textit{asymptotic and exponential  stability} (both for the practical stability and for the stability in the sense of Lyapunov) of $x^*$ for system~\eqref{int_n} with the class of controls given by~\eqref{class}. Although practical asymptotic stability can be
proven with other methods (see Theorem~\ref{thm_pr}), our result does not require the $C^2$-assumption. Furthermore, its proof presents a constructive procedure for defining $\varepsilon$ in~\eqref{cont_eps}.
\begin{theorem}~\label{thm_class}
{Assume that  the cost function $J\in C^2(\mathbb R^n;\mathbb R)$ satisfies A1--A2, the functions $F_{1i},F_{2i}$ in~\eqref{int_n} satisfy  assumption A3 in  $D=B_\Delta(x^*)$ ($0<\Delta\le +\infty$) and the relation~\eqref{class} with some $F_{0i}$, for
each $i\in\{1,\dots,n\}$. Then the following statements hold.\\
I. Let $F_{0i}(J(x))=\alpha$, where $\alpha$ is a positive constant. Then $x^*$ is
 practically exponentially stable for~\eqref{int_n} if $m_1=1$, and $x^*$ is  practically asymptotically stable for~\eqref{int_n} if $m_1>1$.
\\ Namely,
{for any $\delta\in\Big(0,\sqrt[2m_1]{\tfrac{\gamma_1}{\gamma_2}}\Delta\Big)$, $\bar\lambda\in(0,\alpha\kappa_1)$, $\rho\in(0,\delta)$} there exists an $\bar\varepsilon{>}0$ such that, for any $\varepsilon{\in}(0,\bar\varepsilon]$,
$\lambda\in(0,\bar\lambda]$, the solutions of system~\eqref{int_n} with $x^0{\in} B_{\delta}(x^*)$ and $u_{1i}^{\varepsilon}(t),u_{2i}^{\varepsilon}(t)$ defined by~\eqref{cont_eps}, $i=\overline{1,n}$, satisfy}
\begin{align}\label{decay}
&\|x(t)-x^*\|\le \sigma(t)\sqrt[m_1]{\tfrac{\gamma_2}{\gamma_1}}\|x(t)-x^*\|\varphi_{\tilde m}(\lambda (t-\varepsilon))+\rho,\\
&\text{where }\sigma(t){\le}\Big(1{+}\frac{ M}{ L}\Big(\tfrac{\gamma_2}{\gamma_1}\Big)^{\tilde m/2}\delta^{m_1\tilde m}(e^{\nu L\varepsilon}{-}1)\Big),\text{ for all }t\ge0,\nonumber\\
&\sigma(t){\to}1\text{ as }t{\to}\infty,\;\tilde m=1-\tfrac{1}{m_1},\nonumber
\end{align}
\begin{gather}\label{varphi}
\varphi_{\tilde m}(s)=
\left\{\begin{aligned}
&  e^{-\frac{s}{2}}, \text{ if }\tilde m=0, \\
&\Big(1{+}{\tilde m}{s}J^{\tilde m}(x^0)\Big)^{-\tfrac{1}{2m_1\tilde m}},\text{ if }\tilde m>0.
\end{aligned}\right.
\raisetag{10mm}
\end{gather}
II. Let $F_{0i}(J(\cdot)),F_{1i}(J(\cdot)),F_{2i}(J(\cdot))$ satisfy A4. Then $x^*$ is   exponentially stable for~\eqref{int_n} if $\tilde m=1+m_2-\tfrac{1}{m_1}=0$, and $x^*$ is asymptotically stable for~\eqref{int_n} if $\tilde m>0$.\\ Namely,
{ for any $\delta\in\Big(0,\sqrt[2m_1]{\tfrac{\gamma_1}{\gamma_2}}\Delta\Big)$, $\bar\lambda\in(0,\alpha_1\kappa_1)$,} there exists an $\bar\varepsilon{>}0$ such that, for any $\varepsilon{\in}(0,\bar\varepsilon]$, $\lambda\in(0,\bar\lambda]$, the
solutions of system~\eqref{int_n} with $x^0{\in} B_{\delta}(x^*)$ and $u_{1i}^{\varepsilon}(t),u_{2i}^{\varepsilon}(t)$ defined by~\eqref{cont_eps}, $i=\overline{1,n}$, satisfy
 property~\eqref{decay} with $\rho=0$.
\end{theorem}
The proof is in Appendix~\ref{proofs}.  It represents a constructive procedure for choosing  $\bar\varepsilon$ and contains some auxiliary results which may be used
in related problems concerning stabilization and stability analysis.
\begin{remark}
One of the main assumptions required for  asymptotic stability in the sense of
  Lyapunov is $F_{si}(J^*){=}0$ ($s{=}1,2$, $i{=}\overline{1,n}$.) If the above assumption is not satisfied, the proposed result states \textit{practical} asymptotic or exponential stability conditions which do not require the knowledge of $x^*$ and $J(x^*)$, but just  the local behavior of $J$ in a neighborhood of $x^*$.
   Note that only the existence of positive constants in A2 is essential for the assertion of Theorem~3, while we do not require the knowledge of their exact values in the stability proof. However, a crucial point of our construction is that the value of $\bar\varepsilon$ and the decay rate estimates are obtained explicitly in terms of these constants.
 We believe that this theoretical result is conceptually valuable. In practice, this result can be used for
adjusting  the control parameters in order to ensure better convergence properties, provided that the corresponding
information on the cost is available.
\end{remark}
 To prove the {exponential and} asymptotic stability in the sense of Lyapunov, the value $J^*$  (but not $x^*$ itself) has to be known so that $F_{si}$ can be chosen appropriately. The knowledge of $J^*$ may seem quite
restrictive in the context of extremum seeking, however,  as discussed in Section~3.2 and Section~\ref{appl},
such cases are still of relevance in applications. {Besides, if $J^*$ is unknown, often it is possible to choose an acceptable minimal  value  $\hat J>J^*$ and to use   $F_{si}(J-\hat J)$  instead of $F_{si}(J-J^*)$. Although in this case only the practical asymptotic stability can be proven, the controls satisfying $A4$ still may exhibit better performance if $\hat J$ is close enough to $J^*$.}


 \section{Examples}~\label{appl}
 \subsection{Extremum seeking}

As it has  been already mentioned, formula~\eqref{class} describes a whole class of vector fields in~\eqref{int} with various properties for approximating  gradient-like flows of the cost function.  In this section, we illustrate the behavior of solutions
of~\eqref{int_n} with different vector fields discussed in Section~\ref{Main} and controls of the type~\eqref{cont_eps}. For numerical simulation,  we take $J_1(x){=}2(x{-}x^*)^2$, $x\in\mathbb R$,  $x^*{=}1$, $J^*{=}0$, $k_1=1$, and $\varepsilon=0.1$ in each example.
The extremum seeking system introduced in~\cite{Due-Sta-Ebe-Joh-12} is useful in practical implementations due to its simple form:
 \begin{equation}\label{DE}
 \dot x=J_1(x)u_{1}^\varepsilon(t)+u_{2}^\varepsilon(t).
\end{equation}
It can be used for minimizing the cost functions of rather general form, without any information about its analytical expression and extremum values.
The same property holds for the control strategy with so-called bounded updated rates proposed in~\cite{Sch14}:
 \begin{equation}\label{Kr}
   \dot x=\sin(J_1(x))u_{1}^\varepsilon(t)+\cos(J_1(x))u_{2}^\varepsilon(t).
 \end{equation}
 \begin{figure*}[t]
    \includegraphics[width=0.495\linewidth]{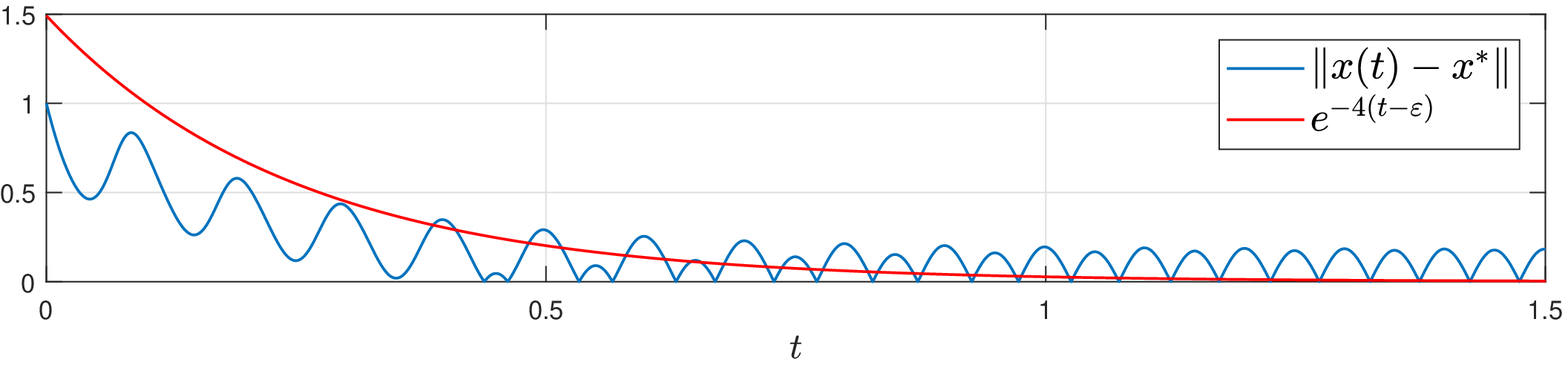}  \includegraphics[width=0.495\linewidth]{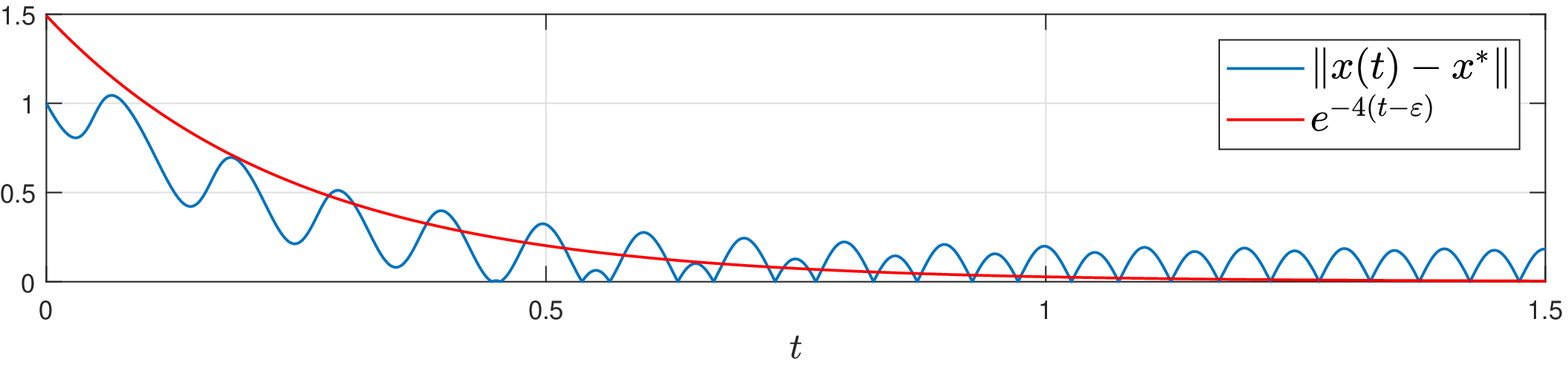}\\
    \includegraphics[width=0.495\linewidth]{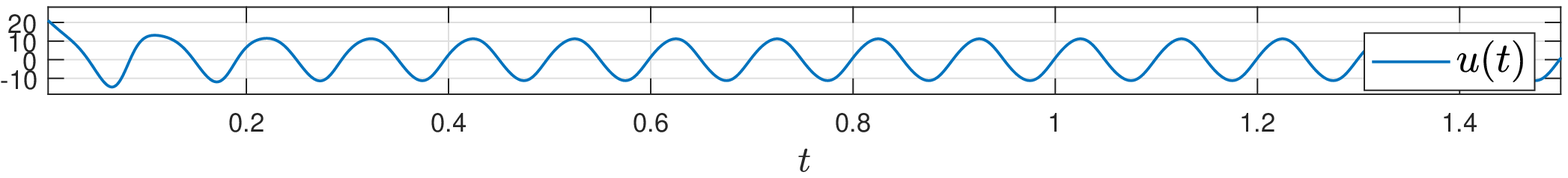}  \includegraphics[width=0.495\linewidth]{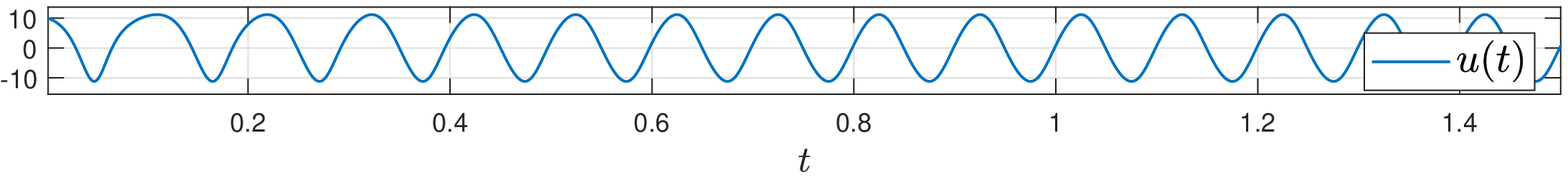}\\
    \begin{minipage}[c][0.2em]{0.48\linewidth}
 \centering{\small a)}
\end{minipage}
    \begin{minipage}[c][0.2em]{0.48\linewidth}
 \centering{\small b)}
\end{minipage}
     \includegraphics[width=0.495\linewidth]{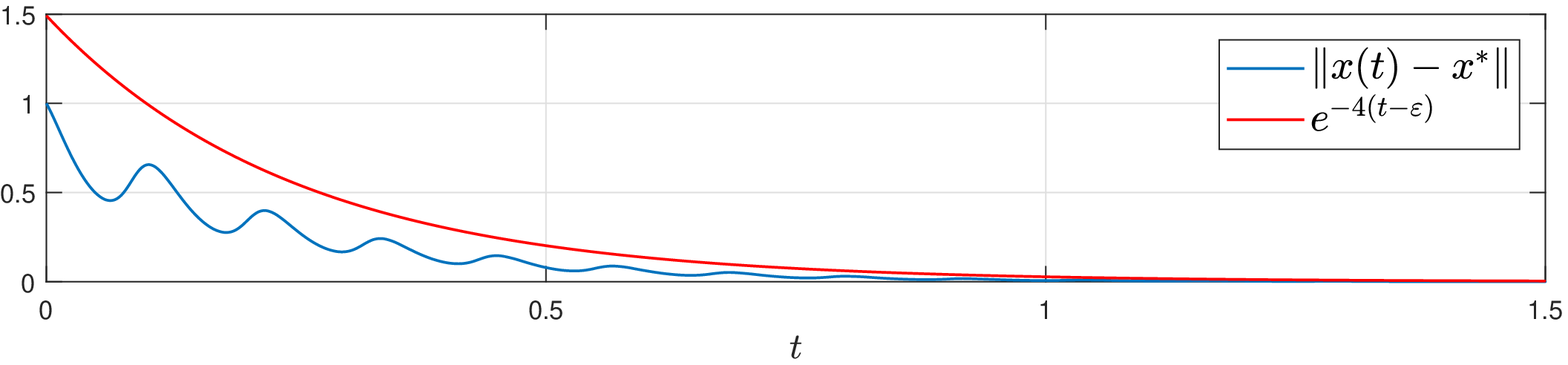} \includegraphics[width=0.495\linewidth]{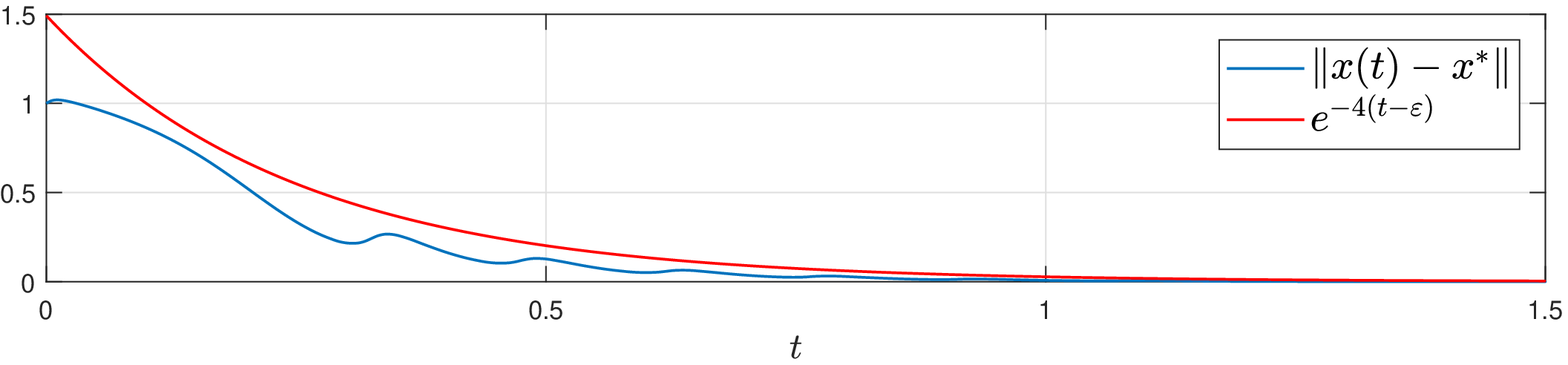}\\
     \includegraphics[width=0.495\linewidth]{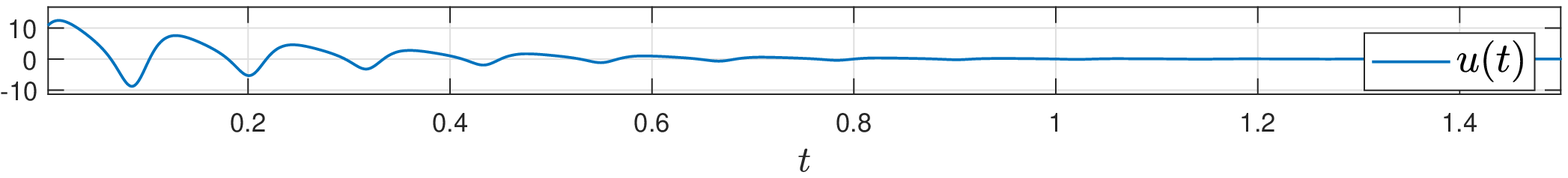} \includegraphics[width=0.495\linewidth]{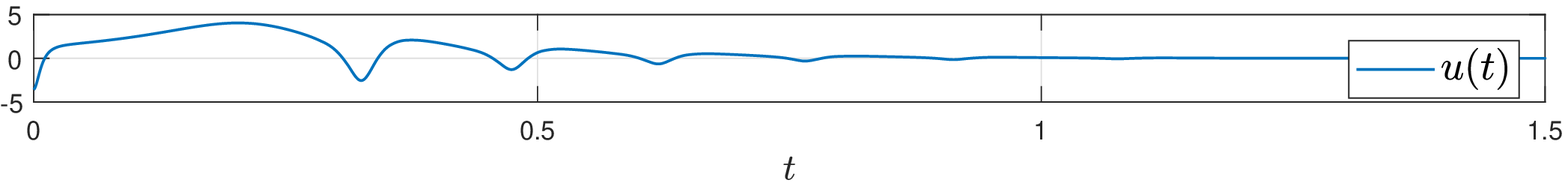}\\
         \begin{minipage}[c][0.2em]{0.48\linewidth}
 \centering{\small c)}
\end{minipage}
    \begin{minipage}[c][0.2em]{0.48\linewidth}
 \centering{\small d)}
\end{minipage}
 \caption{a)--d): Trajectory (top) and control (bottom) of system~\eqref{DE}--\eqref{V2} with $J_1$.}~\label{set1}
\end{figure*}
It is easy to see that both of the above strategies  do not vanish at the extremum point which leads to an oscillating behavior (see Fig.~\ref{set1} a) and b)).
For problems with known  value of the extremum (but not the extremum point),  it is possible to achieve vanishing oscillations
 as $x(t)\to x^*$, as it is stated in Theorem~\ref{thm_class}.
In particular, the following control strategy proposed in~\cite{SD17}  ensures the exponential convergence to  $x^*$:
\begin{equation}\label{Ra}
\begin{aligned}
\dot x&=\sqrt{J_1(x)}\sin\left(\ln (J_1(x))\right)u_{1}^\varepsilon(t)\\
&+\sqrt{J_1(x)}\cos\left(\ln (J_1(x))\right)u_{2}^\varepsilon(t),
\end{aligned}
 \end{equation}
 for $J_1(x)\ne0$, and $\dot x=0$ for $J_1(x)=0$. Indeed, in this case the conditions of Theorem~\ref{thm_class}.II are satisfied with $m_1{=}1,m_2{=}0$.
  \begin{figure}[t]
     \centering
      \includegraphics[width=1\linewidth]{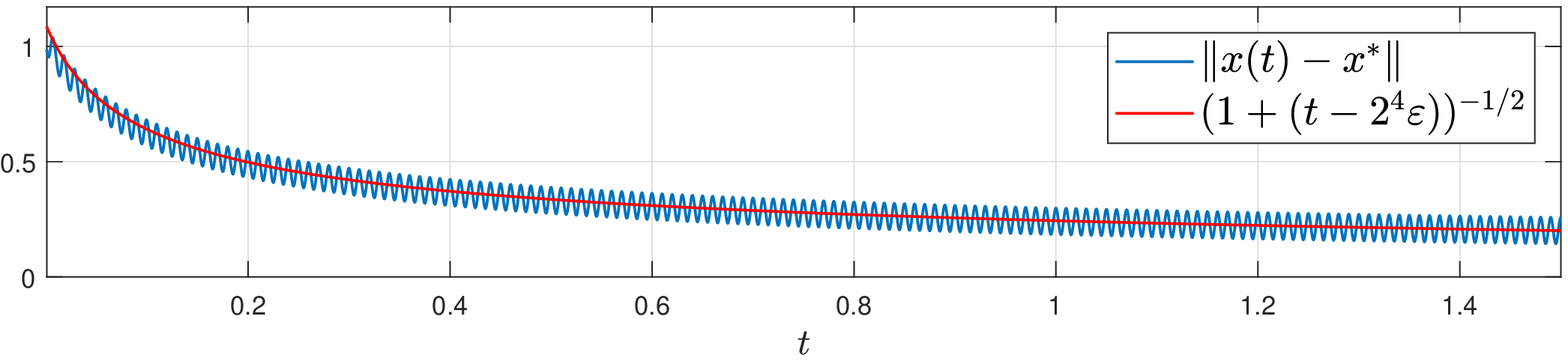}\\
     {\small a)}\\
     \includegraphics[width=1\linewidth]{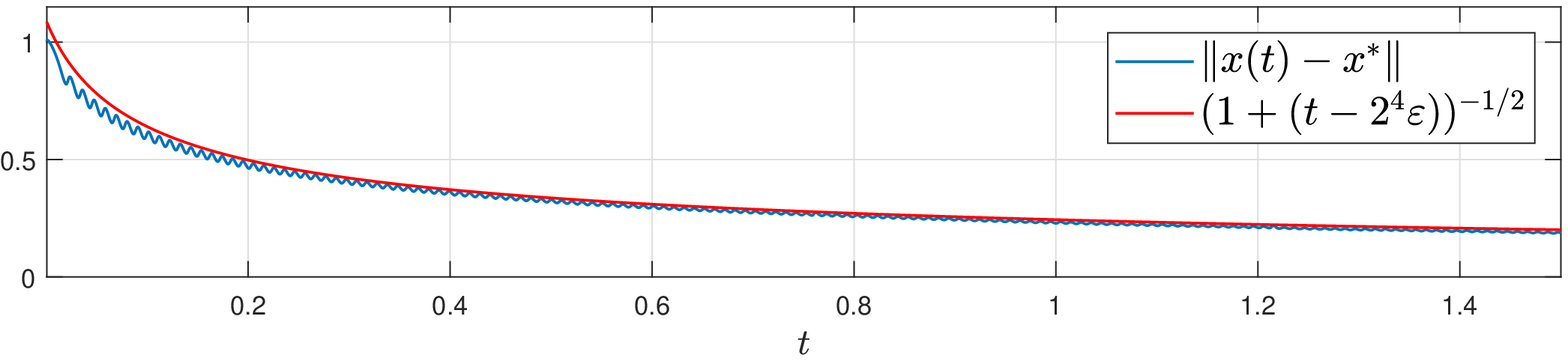}\\
      {\small b)}\\
      \caption{a)--b): Trajectory of systems~\eqref{Kr} and \eqref{V2} with $J_2$.}~\label{set3}
       \end{figure}
 In order to have also bounded update rates, we propose the following extremum seeking system:
  \begin{equation}\label{V2}
\begin{aligned}
\dot x {= }\sqrt{\tfrac{1{-}e^{{-}J_1(x)}}{1{+}e^{J_1(x)}}}&\big(\sin(e^{J_1(x)}{+}2\ln(e^{J_1(x)}{-}1))u_{1}^\varepsilon(t)\\
 {+}&\cos(e^{J_1(x)}{+}2\ln(e^{ J_1(x)}{-}1))u_{2}^\varepsilon(t)\big),
\end{aligned}
 \end{equation}
 for $J_1(x)\ne 0$, and $\dot x=0$ for $J_1(x)=0$. Similarly to the previous example, its vector fields  {locally} satisfy the assumptions of Theorem~\ref{thm_class}.II with $m_1=1$, $m_2=0$.
 Figs.~\ref{set1} c) and d) illustrate the behavior of trajectories of systems \eqref{Ra} and~\eqref{V2}. {The time plots of control functions illustrate that the magnitude of controls satisfying A4 decreases when the cost function approaches the minimum. }
\\
To illustrate the decay rate estimate obtained in Theorem~\ref{thm_class}, consider the function $\varphi(\bar\lambda(t-\varepsilon))$ defined by~\eqref{varphi} with $\bar\lambda=\alpha_1\kappa_1$. In all above  cases, $\alpha_1=8$, $\gamma_1=\gamma_2=\kappa_1=1$, $\varphi(s)=e^{-0.5s}$. Fig.~\ref{set1} demonstrates that estimate~\eqref{decay} holds with good accuracy.
For comparison, consider also $J_2(x)=2(x{-}1)^4$. In this case,  $\varphi(s)=(1+0.5\sqrt2s\|x^0-1\|^2)^{-1/2}$.
 Fig.~\ref{set3} shows the trajectories of systems~\eqref{Kr} and~\eqref{V2} with $\varepsilon=0.01<(\lambda\tilde m \sqrt{J_2(x^0)})^{-1}$, and the time plot of $\varphi$. Observe   that the higher order nonlinearity of $J_2$ results in a slower convergence of the extremum seeking algorithm in comparison with the quadratic $J_1$. This decay rate can be increased, e.g., by choosing $\alpha_1J_2^{-1/2}(x){\le }F_0(J(x)){\le}\alpha_2J_2^{-1/2}(x)$, $\alpha_1,\alpha_2>0$. 
 {\begin{remark}
   In this paper, we do not discuss tuning rules, however, formula~\eqref{class} provides a  possibility to adjust the control parameters. For example, let $u=\sqrt\omega\big(F_1(J(x))\cos(\omega t)+F_2(J(x))\sin(\omega t)\big)$. Taking $F_1(J(x))=c_1 J(x)$, $F_0\equiv c_2$, $\displaystyle F_2(z)=-F_1(z)\int\frac{c_2}{F_1^2(z)}dz\Big|_{z=J(x)}$, we obtain the extremum seeking control $u=c_1\sqrt\omega\big(J(x)\cos(\omega t)+c_2\sin(\omega t)\big)$ with  tuning parameters $c_1,c_2$.
 \end{remark}}
 \subsection{Vibrational stabilization}
 Another application, where  formula~\eqref{class} is of use, is found in the area of the vibrational stabilization
 of systems with partially unknown dynamics. Consider the system
 \begin{equation}\label{sys_vib}
   \dot x=f(x)+g(x)u,
 \end{equation}
 where $x{\in}\mathbb R^n$, $f,g{\in} C^2(\mathbb R^n;\mathbb R^n)$, $u{\in}\mathbb R$.
 It was shown in~\cite{ME,SK13} that, under appropriate assumptions, system~\eqref{sys_vib} can be practically stabilized by using the control law
\begin{equation}\label{cont_vib}
 u=V(x)\sqrt\omega\cos(\omega t)+2\alpha\sqrt\omega\sin(\omega t),
\end{equation}
   \begin{figure}[t]
    \centering
     \includegraphics[width=1\linewidth]{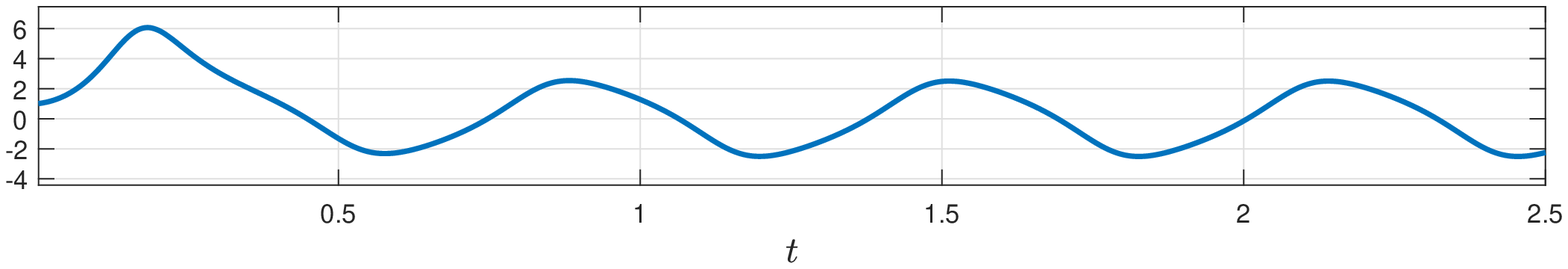}\\
     \includegraphics[width=1\linewidth]{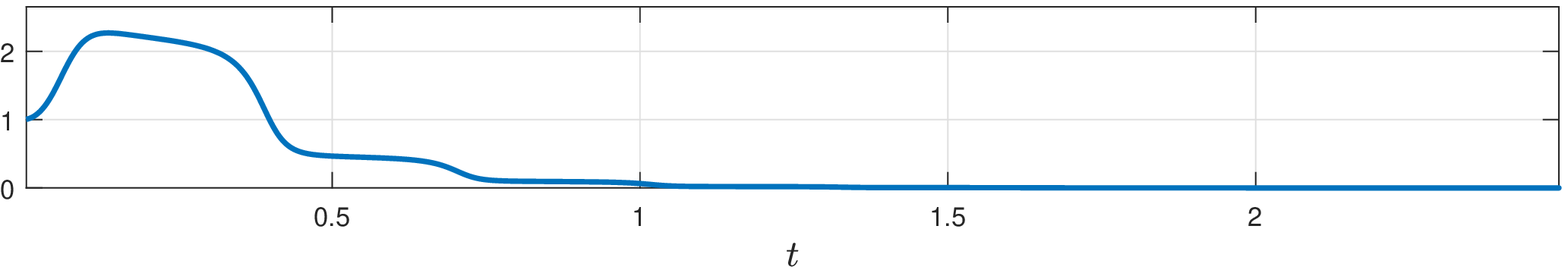}\\
     \includegraphics[width=1\linewidth]{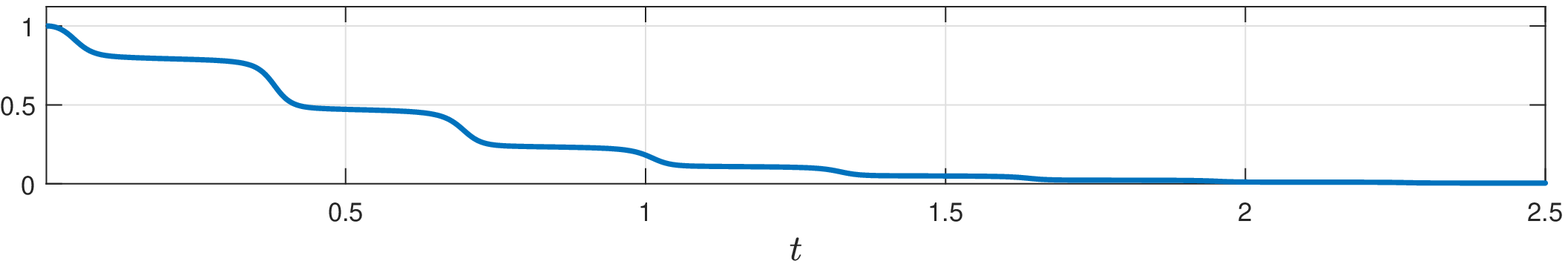}\\
     \caption{Trajectory of~\eqref{ex_vib} with controls as in~\eqref{DE} (top), \eqref{Ra} (middle), \eqref{V2} (bottom).}~\label{fig_vibro}
\end{figure}
 where $\alpha$ is a positive constant, and $V$ is a control Lyapunov function for~\eqref{sys_vib}.
To see why this is possible, compute the corresponding Lie bracket system which takes the form
\begin{equation}\label{lie_vib}
 \dot{\bar x}=f(\bar x)-\alpha g(\bar x)L_gV(\bar x),
\end{equation}
where $L_gV=\nabla V^Tg$.
Hence, \eqref{cont_vib} approximates the control law $u_{L_gV}(x)=-\alpha L_gV(x)$, which is sometimes called damping- or $L_g V$-control law.
An interesting feature of the ``vibrational'' control law \eqref{cont_vib} is
that it only relies on the values of the control Lyapunov function, and neither the vector field $g$
nor the gradient of $V$ is needed to implement this control law. Such controls find many applications, e.g., in adaptive control~\cite{Nus,SK13}.
 Similarly to Section~\ref{Main}, we can construct more general control laws of the form
\begin{equation}\label{class_vib}
 u=F_1(V(x))\sqrt\omega u_1(\omega t)+ 2\alpha F_2(V(x))\sqrt\omega u_2(\omega t),
\end{equation}
 where $F_1,F_2$, satisfy relation~\eqref{class} with $F_0(z)= \alpha>0$, and $u_1,u_2$ satisfy the assumptions
made in Theorem~\ref{thm_appr}.
It is easy to verify that the corresponding Lie bracket system for \eqref{class_vib} coincides with~\eqref{lie_vib},
so that formula~\eqref{class} allows to define a class of vibrational control laws which
approximate the $L_gV$-control laws and stabilize nonlinear systems of the form~\eqref{sys_vib}
using only the values of the control Lyapunov function $V$.  {The  approaches for constructing control  Lyapunov functions in case of unknown $f$, $g$ are proposed, e.g., in~\cite{SK13}.}
Notice that the control Lyapunov functions are positive definite and hence
predestinated to apply formulas with bounded update rate and vanishing amplitudes
as discussed in Section~\ref{Main}.
For a simple illustration, consider the equation
\begin{equation}\label{ex_vib}
  \dot x= x+\mu u,
\end{equation}
where $x\in\mathbb R$ and $\mu\in\mathbb R$ is an unknown parameter, $|\mu|\ge1$. We take the control Lyapunov function $V(x)=x^2$, and $u_{L_gV}(x)=-2\alpha\mu x$, $\alpha>0.5$.
 The evolution of the solution  of system~\eqref{ex_vib} with the control law~\eqref{cont_vib} and  controls of the type~\eqref{Ra}, \eqref{V2},  and the initial condition $x(0)=1$ is presented on Fig.~\ref{fig_vibro}. 

\section{Conclusions}
In this paper, we have proposed a new formula for constructing a class of vector fields to approximate gradient-like flows
based on the Lie bracket approximation idea.  We have shown how this formula gives rise to a broad class of controls for the extremum seeking and vibrational stabilization problems.
It generalizes and unifies some existing results and gives an opportunity for the design of new control functions. While the formula looks rather simple, we believe that  it
potentially comprises more applications than the ones  discussed in this paper. In particular, although we assume the extremum seeking system to have the single integrator dynamics, it is also possible to apply the obtained formula for systems with more
complicated dynamics, e.g. using the approach proposed in~\cite{Durr17}. Besides, this result is of  use for the vibrational stabilization problems with known control Lyapunov functions.
Furthermore, from a conceptual point of view, 
we have presented a novel approach to the proof of  stability properties of  extremum seeking systems. This approach gives several advantages compared  to the existing results. First, the proofs of the main results present a constructive procedure for
defining the frequencies of the control functions for ensuring the practical asymptotic stability; second, {the practical \textit{exponential} stability is proven for certain cost functions}. Finally, the main advantage of the developed approach are
conditions for the asymptotic and exponential stability \textit{in the sense of Lyapunov} for extremum seeking systems whose vector fields satisfy certain additional requirements. An important step in the proof of this result concerns novel decay rate
estimates for the cost function along the solutions of the obtained extremum seeking system. Besides, some auxiliary results of this paper (in particular, Lemmas~2--5) extend the results of~\cite{ZuSIAM, ZGB} and can be exploited
in other control problems, e.g., asymptotic stabilization of nonholonomic systems when the exponential stabilization is not possible.
\bibliographystyle{plain}        
\bibliography{biblio_es}           

\begin{thebibliography}{10}

\bibitem{Chi07}
M.~Chioua, B.~Srinivasan, M.~Perrier, and M.~Guay.
\newblock Effect of excitation frequency in perturbation-based extremum seeking
  methods.
\newblock {\em IFAC Proceedings Volumes}, 40(5):123--128, 2010.

\bibitem{Durr17}
H.-B. D{\"u}rr, M.~Krsti{\'c}, A.~Scheinker, and C.~Ebenbauer.
\newblock Extremum seeking for dynamic maps using {L}ie brackets and singular
  perturbations.
\newblock {\em Automatica}, 83:91--99, 2017.

\bibitem{Due-Sta-Ebe-Joh-12}
H.~B. D{\"u}rr, M.~S. Stankovi\'{c}, C.~Ebenbauer, and K.~H. Johansson.
\newblock Lie {B}racket {A}pproximation of {E}xtremum {S}eeking {S}ystems.
\newblock {\em Automatica}, 49:1538--1552, 2013.

\bibitem{King12}
G.~Gelbert, J.~P. Moeck, C.~O. Paschereit, and R.~King.
\newblock Advanced algorithms for gradient estimation in one- and two-parameter
  extremum seeking controllers.
\newblock {\em J. of Process Control}, 22(4):700--709, 2012.

\bibitem{GrZuNA}
V.~Grushkovskaya and A.~Zuyev.
\newblock Asymptotic behavior of solutions of a nonlinear system in the
  critical case of q pairs of purely imaginary eigenvalues.
\newblock {\em Nonlinear Analysis: Theory, Methods \& Applications},
  80:156--178, 2013.

\bibitem{GrZuIEEE}
V.~Grushkovskaya and A.~Zuyev.
\newblock Optimal stabilization problem with minimax cost in a critical case.
\newblock {\em IEEE Trans. Autom. Control}, 59(9):2512--2517, 2014.

\bibitem{Guay15}
M.~Guay and D.~Dochain.
\newblock A time-varying extremum-seeking control approach.
\newblock {\em Automatica}, 51:356--363, 2015.

\bibitem{Guay03}
M.~Guay and T.~Zhang.
\newblock Adaptive extremum seeking control of nonlinear dynamic systems with
  parametric uncertainties.
\newblock {\em Automatica}, 39(7):1283--1293, 2003.

\bibitem{Har16}
M.~A. Haring and T.~A. Johansen.
\newblock Asymptotic stability of perturbation-based extremum-seeking control
  for nonlinear systems.
\newblock {\em IEEE Trans. Autom. Control}, 62(5):2302--2317, 2017.

\bibitem{Krst03}
M.~Krsti\'{c} and K.~B. Ariyur.
\newblock {\em Real-Time optimization by Extremum Seeking Control}.
\newblock Wiley-Interscience, 2003.

\bibitem{Kr00}
M.~Krsti{\'c} and H.-H. Wang.
\newblock Stability of extremum seeking feedback for general nonlinear dynamic
  systems.
\newblock {\em Automatica}, 36(4):595--601, 2000.

\bibitem{La95}
F.~Lamnabhi-Lagarrigue.
\newblock {V}olterra and {F}liess series expansions for nonlinear systems.
\newblock In W.~S. Levine, editor, {\em The Control Handbook}, pages 879--888.
  1995.

\bibitem{ME}
S.~Michalowsky and C.~Ebenbauer.
\newblock Swinging up the {S}tephenson-{K}apitza pendulum.
\newblock In {\em Proc. 52nd IEEE Conf. on Decision and Control}, pages
  3981--3987, 2013.

\bibitem{Nes09}
D.~Ne{\v{s}}i{\'c}.
\newblock Extremum seeking control: Convergence analysis.
\newblock In {\em Proc. 2009 European Control Conf.}, pages 1702--1715, 2009.

\bibitem{Nes10}
D.~Ne{\v{s}}i{\'c}, Y.~Tan, W.~H. Moase, and C.~Manzie.
\newblock A unifying approach to extremum seeking: Adaptive schemes based on
  estimation of derivatives.
\newblock In {\em Prc. 49th IEEE Conf. on Decision and Control}, pages
  4625--4630, 2010.

\bibitem{Nus}
R.~D. Nussbaum.
\newblock Some remarks on a conjecture in parameter adaptive control.
\newblock {\em Systems \& Control Letters}, 3(5):243--246, 1983.

\bibitem{SK13}
A.~Scheinker and M.~Krsti{\'c}.
\newblock Minimum-seeking for {CLF}s: Universal semiglobally stabilizing
  feedback under unknown control directions.
\newblock {\em IEEE Trans. Autom. Control}, 58(5):1107--1122, 2013.

\bibitem{Sch14}
A.~Scheinker and M.~Krsti\'{c}.
\newblock Extremum seeking with bounded update rates.
\newblock {\em Systems \& Control Letters}, 63:25--31, 2014.

\bibitem{Sch14b}
A.~Scheinker and M.~Krsti{\'c}.
\newblock Non-{C}2 {L}ie bracket averaging for nonsmooth extremum seekers.
\newblock {\em Journal of Dynamic Systems, Measurement, and Control},
  136(1):011010--1--011010--10, 2014.

\bibitem{Sch16}
A.~Scheinker and D.~Scheinker.
\newblock Bounded extremum seeking with discontinuous dithers.
\newblock {\em Automatica}, 69:250--257, 2016.

\bibitem{SD17}
R.~Suttner and S.~Dashkovskiy.
\newblock Exponential stability for extremum seeking control systems.
\newblock {\em Preprints of the 20th IFAC World Congress}, pages 16034--16040,
  2017.

\bibitem{review}
Y.~Tan, W.~H. Moase, C.~Manzie, D.~Ne\v{s}i\'{c}, and I.~M.~Y. Mareels.
\newblock Extremum seeking from 1922 to 2010.
\newblock In {\em Proc. 29th Chinese Control Conf.}, pages 14--26, 2010.

\bibitem{TNM08}
Y.~Tan, D.~Ne{\v{s}}i{\'c}, and I.~Mareels.
\newblock On the choice of dither in extremum seeking systems: A case study.
\newblock {\em Automatica}, 44(5):1446--1450, 2008.

\bibitem{ZuSIAM}
A.~Zuyev.
\newblock Exponential stabilization of nonholonomic systems by means of
  oscillating controls.
\newblock {\em SIAM J. on Control and Optimization}, 54(3):1678--1696, 2016.

\bibitem{ZGB}
A.~Zuyev, V.~Grushkovskaya, and P.~Benner.
\newblock Time-varying stabilization of a class of driftless systems satisfying
  second-order controllability conditions.
\newblock In {\em Proc. 15th European Control Conf.}, pages 1678--1696, 2016.

\end{thebibliography}
\newpage
\begin{appendix}
\section{Proofs}~\label{proofs}
\subsection{Preliminary results}
Without loss of generality, throughout this section we assume $J^*=0$.
An important step of the proof is the representation of solutions of system~\eqref{int_n} with initial conditions $x(0)=x^0\in D\subseteq\mathbb R^n$  by using the Volterra series~\cite{La95,ZuSIAM}. Before proving   Theorem~\ref{thm_class}, we need
to state several auxiliary results. Since they can be used not only in the extremum seeking problem, we formulate them for a general system
 \begin{equation}\label{nonh}
 \dot x=\sum_{i=1}^{\ell}f_i(x)u_i(t), \quad x\in D\subseteq\mathbb R^n, \;f_i:D\to\mathbb R^n.
 \end{equation}
\begin{lemma}\label{volterra}
{ Let the vector fields $f_i$ be Lipschitz continuous in a domain $D\subseteq\mathbb R^n$, and $f_i\in C^2(D\setminus\Xi;\mathbb R)$, where $\Xi=\{x\in D:f_i(x)=0\text{ for all }1\le i\le \ell\}$. Assume, moreover, that $L_{f_j}f_i,L_{f_l}L_{f_j}f_i\in
C(D;\mathbb R^n)$,   for all $i,j,l=\overline{1,\ell}$. If $x(t)\in D$, $t\in[0,\tau]$, is a  solution of system~\eqref{nonh} with $u\in C([0,\tau];\mathbb R^m)$ and $x(0)=x^0\in D$, then  $x(t)$ can be represented by the Volterra series:}
\begin{align}
   & x(t){=}x^0{+}{\sum_{i=1}^{\ell}}f_i(x^0)\int\limits_0^t u_{i}(v)dv  {+}\sum_{\hspace{-0.75em}i,j=1}^{\ell}L_{f_j}f_i(x^0)\int\limits_0^t\int\limits_0^v  u_{i}(v) u_{j}(s)dsdv+R(t),\,t\in[0,\tau]\label{volt1}\\
&\text{where    }R(t){=}{\sum\limits_{\hspace{-0.75em}i,j,l=1}^{\ell}}{\int\limits_0^t}{\int\limits_0^v}{\int\limits_0^s}L_{f_l}L_{f_j}f_i(x(p))  u_{i}(v)u_{j}(s)u_{l}(p)dpdsdv \nonumber
    \end{align}
     is the remainder of the Volterra series expansion.
     \end{lemma}
\begin{proof}
  The validity of~\eqref{volt1} is justified, e.g., in~\cite{La95} for analytic vector fields $f_i$. Recall that $f_i$ are Lipschitz continuous in $D$, and $u\in C([0,\tau];\mathbb R^m)$, therefore, $u$ is bounded in on $[0,\tau]$. Hence, the uniqueness
  of the solution to the Cauchy problem implies that the set $\Xi$ is strongly invariant under our assumptions, i.e. if $x(t)$ is a solution of~\eqref{nonh}  with some control $u\in C([0,\tau];\mathbb R^m)$ and $x(t^*)\in\Xi$ for some $t^*\in[0,\tau]$
  then $x(t)\in\Xi$ for all $t\in[0,\tau]$ (and $x(t)\equiv{\rm const}$ because of the definition of $\Xi$). These arguments show that either $x(t)\in\Xi$ (so that the representation~\eqref{volt1} is valid with all $f_i(x^0)$, $L_{f_j}f_i(x^0)$,
  $L_{f_l}L_{f_j}f_i(x(p))$ being zero), or $x(t)\notin\Xi$ for all $t\in[0,\tau]$. In the latter case we apply the fundamental theorem of calculus to see that
\begin{align*}
f_i(x(v))&=f_i(x^0){+}\int_0^v\frac{d f_i(x(s))}{d s}ds\\
&=f_i(x^0){+}\int_0^v\sum_{j=1}^m L_{f_j}f_i(x(s))u_j(s)ds.
\end{align*}
 Applying the same procedure to $L_{f_j}f_i(x(s))$ and  representing $x(t){=}x^0{+}{\sum\limits_{i=1}^{\ell}}{\int_0^t}f_i(x(v))u_i(v)dv$, we finally obtain~\eqref{volt1} for all solutions of system~\eqref{nonh} in $D\setminus\Xi$.
\end{proof}
 \begin{lemma}~\label{lemma_x}
 Let $D\subseteq\mathbb R^n$, $x^*\in D$, $x(t)\in D$, $0\le t\le\tau$, be a  solution of system~\eqref{nonh}. Assume that there exist  $m\ge0$, $\tilde M ,L>0$ such that
\begin{align*}
\|f_i(x)\|\le  \tilde M\|x-x^*\|^m,\,
\|f_i(x)-f_i(y)\|\le  L\|x-y\|,
\end{align*}
for all $x,y\in D$, $i{=}\overline{1,\ell}$.
    Then
    \begin{equation}\label{est_Z}
      \|x(t)-x(0)\|{\le}\tfrac{\tilde M\|x(0)-x^*\|^m}{ L}(e^{\nu Lt}{-}1),\;t{\in}[0,\tau],
    \end{equation}
    with $\nu=\max\limits_{t\in[0,\tau]}\sum_{i=1}^{\ell}|u_{i}(t)|$.
  \end{lemma}
  The proof  is analogous to the proof of~\cite[Lemma 4.1]{ZuSIAM}. {Note that Lemma~\ref{lemma_x} complements \cite[Lemma 4.1]{ZuSIAM} stating that  $\|x(t)-x(0)\|\to 0$  as $x(0)\to x^*$.}
  \begin{lemma}~\label{lemma_r}
{   Let the conditions of Lemma~\ref{lemma_x} be satisfied. Assume that $\nu\tau<1$ and there exist $\tilde H\ge0$, $\varpi\in\{0\}\cup[1,\infty)$ such that
$$
{\sum\limits_{\hspace{-0.75em}i,j,l=1}^{\ell}}\|L_{f_l}L_{f_j}f_i(x)\|\le \tilde H\|x-x^*\|^{\varpi}\text{ for all } x\in D.
$$
Then, for all $t\in[0,\tau]$, the remainder $R(t)$ of the Volterra expansion~\eqref{volt1} of $x(t)$ satisfies the estimate $$\|R(t)\|\le (t\nu)^3\|x(0)-x^*\|^\varpi C(t\nu,  \|x(0){-}x^*\|),$$
$C(t\nu, \|x(0)-x^*\|){=}2^{\varpi{-}1}\tilde H\big(1{+}c_1(\nu t \|x(0){-}x^*\|^{m-1})^\varpi\big)$, \\ $c_1=\tfrac{6(\tilde M(e^L-1))^{\varpi}}{L(\varpi{+}1)(\varpi{+}2)(\varpi{+}3)}$.}
  \end{lemma}
\begin{proof}
{ From~\eqref{volt1} and Lemma~\ref{lemma_x},
\begin{align*}
&\|R(t)\|\le \nu^3{\sum\limits_{\hspace{-0.75em}i,j,l=1}^{\ell}}{\int\limits_0^t}{\int\limits_0^\tau}{\int\limits_0^s}\bigg\|L_{f_l}L_{f_j}f_i(x(p))\bigg\|dpdsd\tau\\
& \le  \nu^3H{\int\limits_0^t}{\int\limits_0^\tau}{\int\limits_0^s}\|x(p)-x^*\|^{\varpi}dpdsd\tau\\
& \le  \nu^3H2^{\varpi{-}1}{\int\limits_0^t}{\int\limits_0^\tau}{\int\limits_0^s}\big(\|x(0){-}x^*\|^{\varpi}{+}\|x(p){-}x(0)\|^{\varpi}\big)dpdsd\tau\\
&\le \nu^3H2^{\varpi{-}1}\|x(0){-}x^*\|^{\varpi}{\int\limits_0^t}{\int\limits_0^\tau}{\int\limits_0^s}\big(1{+}\Big(\tfrac{\tilde M}{L}\Big)^\varpi\|x(0){-}x^*\|^{\varpi(m-1)}(e^{\nu Lt}{-}1)^\varpi\big)dpdsd\tau.
\end{align*}
The estimation of $e^{\nu Lt}{-}1\le \nu t(e^L-1)$ for $\nu t\le 1$ and computation of the integrals complete the proof.}
\end{proof}
\begin{lemma}~\label{lemma_v}
    Let $D\subseteq\mathbb R^n$ be a bounded convex domain, $V,h_i:D\to\mathbb R$, $i=\overline{1,n}$, $V\in C^2(D;\mathbb R)$, $x^*\in D$, and let the following inequalities hold:
   $$
   \begin{aligned}
    & \gamma_1\|x-x^*\|^{2m_1} \le V(x)\le \gamma_2\|x-x^* \|^{2m_1},\\
    &  \kappa_1 V(x)^{2-\frac{1}{m_1}}\le \|\nabla V(x)\|^2\le\kappa_2 V(x)^{2-\frac{1}{m_1}},\\
    &\left\|\frac{\partial^2 V(x)}{\partial x^2}\right\|\le\mu V(x)^{1-\frac{1}{m_1}},\\
     &  \alpha_1V(x)^{m_2}\le h_i(x)\le \alpha_2 V(x)^{m_2},\text{ for all }x\in D,
    \end{aligned}
    $$
   where $m_1\ge1$, {$m_2\ge \tfrac{1}{m_1}-1$}, and $\alpha_1,\alpha_2,\gamma_1,\gamma_2,\kappa_1,\kappa_2,\mu$ are positive constants.
   Then, for any $x^0=x(0)\in D\setminus\{x^*\}$ and any function $x:[0,\varepsilon]\to D$  satisfying the conditions
    \begin{equation*}
   x(0){=}x^0,\;   x(\varepsilon){=}x^0{-}\varepsilon\sum_{i{=}1}^n\frac{\partial V(x^0)}{\partial x_i}h_i(x^0)e_i{+}r_\varepsilon,\,r_\varepsilon{\in}\mathbb R^n,
    \end{equation*}
{the function $V$ satisfies the estimate:
 \begin{align*}
V(x(\varepsilon)){\le} V(x^0)\Big(1{-}\frac{\varepsilon\varkappa_1}{m_1}V^{\tilde m}(x^0){+}\frac{\varepsilon^2\varkappa_2}{2m_1^2}V^{2\tilde m}(x^0)\Big)^{m_1},
 \end{align*}}
where $\tilde m = 1+m_2-\tfrac{1}{m_1}$, $\varkappa_1=\alpha_1\kappa_1-\frac{\sqrt{\kappa_2}\|r_\varepsilon\|}{\varepsilon V^{\tilde m+\tfrac{1}{2m_1}}(x^0)}$,  \\
$\varkappa_2=((m_1-1)\kappa_2+\mu m_1)\bigg(\alpha_2\sqrt{\kappa_2}+\frac{\|r_\varepsilon\|}{\varepsilon V^{\tilde m+\tfrac{1}{2m_1}}(x^0)}\bigg)^2$.
  \end{lemma}
\begin{proof}
 {
For  $x^0{\in }D\setminus\{x^*\}$, we
  denote  $$y=x(\varepsilon)-x^0=-\varepsilon\sum_{i=1}^n\frac{\partial V(x^0)}{\partial x_i}h_i(x^0)e_i+r_\varepsilon,$$
  and introduce the function $v(\theta)=V^{\tfrac{1}{m_1}}(x^0+\theta y)$, $\theta\in[0,1]$. Note that $x^0+\theta y\in D$ for all $\theta\in[0,1]$ since $D$ is convex, so the above $v(\theta)$ is well-defined.
  Straightforward computations show that
  $$
   v'(\theta)=\frac{1}{m_1}V^{\tfrac{1}{m_1}-1}(x)\sum\limits_{i=1}^n\frac{\partial V}{\partial x_i}y_i\Big|_{x=x^0+\theta y}\text{ if }x^0+\theta y\ne x^*.
  $$
 In case $x^0+\theta^*y=x^*$ for some $\theta^*\in(0,1]$, we see that $ v'(\theta^*)=\lim\limits_{\theta\to\theta^*}\frac{v(\theta)-v(\theta^*)}{\theta-\theta^*}=0$ and $\lim\limits_{\theta\to\theta^*} v'(\theta)=0$ under the assumptions of this lemma. Thus we have shown that $v\in C^1([0,1];\mathbb  R^+)$. The next step is to prove that the function $w(\theta)=v'(\theta)$ is Lipschitz continuous on $[0,1]$. Indeed, if $x^*\notin{\rm co}\{x^0,x(\varepsilon)\}$, then $w\in C^1([0,1];\mathbb  R^+)$, and
\begin{align*}
  w'(\theta)=\left.\frac{V^{\tfrac{1}{m_1}-1}(x)}{m_1}\left(\frac{1-m_1}{m_1V(x)}\Big(\sum\limits_{i=1}^n\frac{\partial V(x)}{\partial x_i}y_i\Big)^2+\sum_{i,j=1}^n\frac{\partial^2 V(x)}{\partial x_i\partial x_j}y_iy_j\right)\right|_{x=x^0+\theta y}.
\end{align*}
  By exploiting the assumptions of this lemma, we conclude that $|w'(\theta)|\le \bar L$ for all $\theta\in[0,1]$, where $\bar L=\frac{(m_1-1)\kappa_2+\mu m_1}{m_1^2}\|y\|^2$.
  If $x^*=x^0+\theta^* y$ for $\theta^*\in(0,1]$, then $w'(\theta)$ is continuous at each point $\theta\in[0,1]\setminus\{\theta^*\}$, and $|w'(\theta)|\le\bar L$ for all $\theta\in[0,1]\setminus\{\theta^*\}$. These arguments imply that $w(\cdot)$ is Lipschitz continuous, so that $w(\theta)\le w(0)+\bar L\theta$ for all $\theta\in[0,1]$. By integrating the above inequality, we get:
  \begin{align*}
    v(\theta)=&v(0)+\int_0^\theta w(s) ds\le v(0)+\int_0^\theta(w(0)+\bar Ls)ds\\
    =&v(0)+v'(0)\theta+\frac{\bar L}{2}\theta^2.
  \end{align*}
  In particular, for $\theta=1$ with regard for the definition of $v$, we have:
    \begin{align*}
        V^{\tfrac{1}{m_1}}(x(\varepsilon))\le V^{\tfrac{1}{m_1}}(x^0){+}\frac{1}{m_1}V^{\tfrac{1}{m_1}{-}1}(x^0)\sum_{i=1}^n\frac{\partial V(x^0)}{\partial x_i}y_i+\frac{\bar L}{2}.
    \end{align*}
    Then the assertion of Lemma~\ref{lemma_v} follows from the above estimate by exploiting the assumptions on $\|\nabla V(x)\|$, $\left\|\frac{\partial^2 V(x)}{\partial x^2}\right\|$, $h_i(x)$, and the definition of $\bar L$.}
   %
\end{proof}
{Lemma~\ref{lemma_v} provides an extension of~\cite[Lemma~3.2]{ZuSIAM} from $\tilde m=0 $ to an arbitrary $\tilde m\ge0$.}
\subsection{{{Proof of Theorem~\ref{thm_class}}}}
{The idea of the proof is based on the approach proposed in~\cite{ZuSIAM,ZGB}. However, unlike the above papers, we use continuous controls and  the {classical notion of  solutions} (Carath\'eodory solutions). Besides, we use more general assumptions on the vector fields and on the cost function.}\\
Let $\Delta\in(0,\infty]$ be such that  A1--A3  are satisfied in $ D=B_\Delta(x^*)$, {$\delta\in\Big(0,\sqrt[2m_1]{\tfrac{\gamma_1}{\gamma_2}}\Delta\Big)$ be fixed}, and let
\begin{equation}\label{nu}
  \nu=\max\limits_{t}\sum_{\underset{s=1,2}{i=1}}^n|u_{si}^\varepsilon(t)|=2\sqrt{2\pi}\varepsilon^{-1/2}\sum\limits_{i=1}^n\sqrt{{k_i}}.
\end{equation}
\textit{Step 1.}
{For an arbitrary  $\delta_0{\in}(\sqrt[2m_1]{\tfrac{\gamma_2}{\gamma_1}}\delta,\Delta)$,} denote
$$
 D_0=\overline{B_{\delta_0}(x^*)}{\subset} D{=}B_{\Delta}(x^*),\,M_F=\sup\limits_{\hspace{-1em}\underset{s=1,2,1\le i\le n}{x\in D_0}}|F_{si}(J(x))|.
$$
First of all, we specify an $\varepsilon_0>0$ such that, for each $\varepsilon\in(0,\varepsilon_0]$,  all solutions $x(t)$ of system~\eqref{int_n} with controls $u^\varepsilon(t)$~\eqref{cont_eps} and the initial conditions $x(0)=x^0\in D_0$ are well defined
on $t\in[0,\varepsilon]$.
We put $d=\Delta-\delta_0>0$ and see that \begin{equation}\label{eps0}
0<\varepsilon_{0}<\Big({2\sqrt{2\pi} L\sum\limits_{i=1}^n\sqrt{k_i}}\Big)^{-2}\ln^2\bigg(\tfrac{L d}{M_F}+1\bigg).
\end{equation}
Then by Lemma~\ref{lemma_x} with $m=0$,
$\|x(t)-x^0\|<d,
$
for each $\varepsilon\in(0,\varepsilon_0]$, $t\in[0,\varepsilon]$.
If $\Delta{=}{+}\infty$, then we take $d{=}{+}\infty$ and arbitrary $\varepsilon_{0}{\in}(0,\infty)$.
Hence,  all solutions $x(t)$ of system~\eqref{int_n} with the initial conditions $x^0{\in }D_0$ and controls $u_{si}^{\varepsilon}(t)$~\eqref{cont_eps} are in the set $D$ for $t\in[0,\varepsilon]$. {Without loss of generality, we assume $x(\varepsilon)\ne x^*$ (otherwise, similarly to the proof of Lemma~\ref{volterra}, we will have that $x(t)\equiv x^*$).} \\
 \textit{Step 2.} Let $J$  satisfy the conditions of the theorem. We introduce the level sets
$$
\mathcal L_c=\{x\in D:J(x)\le c\},
$$
and define
$
{c_0{=}\gamma_1\delta_0^{2m_1}.}
$
It is easy to see that  $\overline{B_{\delta}(x^*)}\subseteq \mathcal L_{c_0}\subseteq D_0$, and $\mathcal L_c\subseteq\mathcal L_{c_0}$ for all $c\le c_0$.\\
We begin with the proof of  assertion II of Theorem~3.
\begin{center}
{   \textit{Proof of exponential and asymptotic stability in the sense of Lyapunov}}
\end{center}
\textit{Step 3.II.} Applying Lemma~\ref{volterra} with formula~\eqref{class}  and the expressions for controls~\eqref{cont_eps}, the representation~\eqref{volt1}  may be written as
 \begin{align}
  x(\varepsilon)&=x^0+\frac{1}{2}\sum_{i=1}^n{[F_{1i}e_i,F_{2i}e_i](J(x^0))}\int\limits_0^\varepsilon\int\limits_0^\tau\big(u_{2i}(\tau)u_{1i}(\theta)-u_{1i}(\tau)u_{2i}(\theta)\big)d\theta d\tau+R(\varepsilon)\nonumber\\
   &=x^0{-}\varepsilon\sum_{i=1}^n\frac{\partial J(x^0)}{\partial x_i}F_{0i}(J(x^0))e_i+R(\varepsilon).\label{volt}
\end{align}
Let us estimate the remainder  $R(\varepsilon)$ in~\eqref{volt}.
Choosing $\varepsilon_1{=}\big(2\sqrt{2\pi}\sum_{i=1}^n\sqrt{k_i}\big)^{-2}$, we guarantee that $\nu\varepsilon\le1$ for all $\varepsilon\in(0,\varepsilon_1]$.
{Note that, as requested in A4, $2m_1m_3=m_1(m_2+1)>0$, $2m_1m_4=3m_1(1+m_2)-2\ge1$.}
Then Lemma~\ref{lemma_r}   with $m=2m_1m_3$, $\varpi=2m_1m_4$, $\tilde M=M\gamma_1^{m_3}$, $\tilde H=H\gamma_1^{m_4}$ implies the following estimate:
$$\|R(\varepsilon)\|{\le} (\varepsilon\nu)^3\|x^0{-}x^*\|^{2m_1m_4}C(\nu\varepsilon,\|x^0{-}x^*\|).$$
Define $\varepsilon_2{\in}(0,\min\{\varepsilon_0,\varepsilon_1,c_0^{-\tilde m}\})$ (for $\tilde m{=}0$ or $\Delta{<}\infty$, the $\varepsilon_2$ can be chosen independently of $c_0$.)
 Then, for $\varepsilon{\in}(0,\varepsilon_2]$, $x^0{\in}\mathcal L_{c_0}$,
\begin{equation}\label{estc0}
\varepsilon J^{\tilde m}(x^0)\le\varepsilon  c_0^{\tilde m}<1.
\end{equation}
Hence,  there is a constant $\Omega>0$ such that
$$
\|R(\varepsilon)\|\le\Omega J^{m_4}(x^0)\varepsilon^{3/2}=\Omega \varepsilon^{\tfrac{3}{2}} J^{\tfrac{3}{2}\tilde m+\frac{1}{2m_1}}(x^0),
$$
for all $\varepsilon\in(0,\varepsilon_2]$, $x^0\in\mathcal L_{c_0}$.
Indeed, from~\eqref{estc0},
 $$J^{m_3-\tfrac{1}{2m_1}}(x^0)=J^{\tfrac{\tilde m}{2}}(x^0)<\varepsilon^{-1/2};$$ then~\eqref{nu} implies that the constant  $C$ in Lemma~\ref{lemma_r} does not depend on $x^0$ and $\varepsilon$:
 \begin{align*}
 C
 \le \gamma_1^{-m_4/2m_1}2^{2m_1m_4{-}2}\tilde H\big(1{+}c_1\gamma_1^{1/2m_1-m_3}(\nu \sqrt{\varepsilon_2})^{2m_1m_4}\big).
 \end{align*}
 \\
\textit{Step 4.II.} Assume that assumption A4 hold, and fix $\bar\lambda\in(0,\alpha_1\kappa_1)$. On this step, {we will find an $\bar\varepsilon>0$ ensuring the following property: for any $\varepsilon\in(0,\bar\varepsilon]$, $\lambda\in(0,\bar
\lambda]$, the  solutions of system~\eqref{int_n} with the controls $u_{is}^\varepsilon(t)$ and initial conditions from $\mathcal L_{c_0}$ satisfy the  inequality:}
\begin{equation}\label{Jest1}
J(x(\varepsilon)){\le} J(x^0)\Big(1{-}\frac{\varepsilon\lambda}{m_1}J^{\tilde m}(x^0)\Big)^{m_1},
\end{equation}
Lemma~\ref{lemma_v} with $V{=}J$,  $m_1$ defined in A3, $m_2\ge \tfrac{1}{m_2}-1$, yields
\begin{equation}\label{Jest}
J(x(\varepsilon)){\le} J(x^0)\Big(1{-}\frac{\varepsilon\varkappa_1}{m_1}J^{\tilde m}(x^0){+}\frac{\varepsilon^2\varkappa_2}{2m_1^2}J^{2\tilde m}(x^0)\Big)^{m_1},
\end{equation}
where $\tilde m {=} 1{+}m_2{-}\tfrac{1}{m_1}$, and $\varkappa_1,\varkappa_2$ are defined in Lemma~\ref{lemma_v}.
 Estimating $\varkappa_1,\varkappa_2$ and taking into account~\eqref{estc0}, we obtain
$$
\begin{aligned}
&J(x(\varepsilon))\le J(x^0)\Big(1- \tfrac{\varepsilon J^{\tilde m}(x^0)}{m_1}\big(\alpha_1\kappa_1-\sqrt{\varepsilon}\lambda_1c_0^{\tilde m/2}\big)\Big)^{m_1},\\
&\lambda_1=\Omega\sqrt{\kappa_2 }+\tfrac{1}{2m_1} ((m_1-1)\kappa_2+\mu m_1)\bigg(\alpha_2\sqrt{\kappa_2}+\Omega\bigg)^2.
\end{aligned}
$$
Recall that $\bar\lambda\in(0,\alpha_1\kappa_1)$ and put $\varepsilon_4\in\Big(0,\Big(\tfrac{\alpha_1\kappa_1-\bar\lambda}{c_0^{\tilde m/2}\lambda_1}\Big)^2\Big]$.
 Defining $\bar\varepsilon{=}\min\Big\{\varepsilon_3,\varepsilon_4,{\tfrac{m_1}{\bar\lambda c_0^{\tilde m}}}\Big\}$,  we conclude that
$
  \tfrac{\varepsilon J^{\tilde m}(x^0)\lambda}{m_1}{<}1
$
and~\eqref{Jest1} is satisfied for any $\varepsilon{\in}(0,\bar\varepsilon]$, $\lambda{\in}(0,\bar\lambda]$, $x^0{\in} \mathcal L_{c_0}{\subseteq} D_0$. Note that the above $\bar\varepsilon$ may be chosen independently of $x^0\in B_\delta(x^*)$.
\\
\textit{Step 5.II.} The next step is to estimate $\|x(t)-x^*\|$ for $t=\varepsilon,2\varepsilon,\dots$.\\
Let $\tilde m=0$. Since $\tfrac{\varepsilon\lambda}{m_1}<1$, estimate~\eqref{Jest1} can be rewritten as $J(x(\varepsilon))\le J(x^0)\Big(e^{-\tfrac{\lambda}{m_1}\varepsilon}\Big)^{m_1}$. Iterating the obtained inequality for $x^0\in\mathcal
L_{c_0}$,
 we get
 $
 J(x(t))\le J(x^0)e^{-\lambda{t}}\text{ for }t=0,\varepsilon,2\varepsilon,\dots.
 $\\
For $\tilde m>0$, recall that 
$\tfrac{\varepsilon J(x^0)^{\tilde m}\lambda}{m_1}<1$, and, additionally, require  $\varepsilon\in(0,({\bar\lambda \tilde m c_0^{\tilde m}})^{-1})$. Then we may rewrite~\eqref{Jest1} as
\begin{align*}
J(x(\varepsilon))&{\le} J(x^0)\Big(1{-}\tfrac{\varepsilon\lambda}{m_1}J^{\tilde m}(x^0)\Big)^{m_1}\\
&{\le}{J(x^0)\Big(1{+}{\tilde m} {\varepsilon\lambda}J^{\tilde m}(x^0)\Big)^{{-}\tfrac{1}{\tilde m}}.}
\end{align*}
Iterating the obtained inequality for $x^0\in\mathcal L_{c_0}$,
 we get
 $$
 J(x(t))\le J(x^0)\Big(1{+}{t\lambda}{\tilde m}J^{\tilde m}(x^0)\Big)^{-\tfrac{1}{\tilde m}}\text{ for }t=0,\varepsilon,2\varepsilon,\dots.
 $$
 Note that, since   $J(x(\varepsilon))\le J(x^0)$,
 the same $\lambda$, $c_0$, $\varepsilon$ can be chosen for all $x^0\in D_0$.\\
 Thus, if $x^0\in B_\delta(x^*)\subseteq\mathcal L_{c_0}$, $\varepsilon\in(0,\bar\varepsilon]$, and $u_{si}^\varepsilon$ are given by~\eqref{cont_eps}, then the corresponding solution of~\eqref{int_n} is well defined in $D$:
$
x(n\varepsilon)\in\mathcal L_{c_0}\subseteq D_0
$
 for $n=0,1,2,\dots,$ and  due to the choice of $d,\varepsilon$ and Lemma~\ref{lemma_x}: $x(t)\in D$ for all $t\ge0$.
Combining the above results with A3, we conclude that
\begin{align*}
  \|x(t)-x^*\|\le \sqrt[2m_1]{\tfrac{\gamma_2}{\gamma_1}}\|x^0-x^*\|\varphi_{\tilde m}(\lambda t),\text{ for }t=0,\varepsilon,2\varepsilon,\dots,
\end{align*}
\textit{Step 6.II.} Finally, it remains to prove the exponential (or power) decay rate of $\|x(t)-x^*\|$ for all $t\ge 0$.\\
For any $t\ge 0$  denote the integer part of $\frac{t}{\varepsilon}$ as $ t_{in}^\varepsilon$, and note that  $0\le t-t_{in}^\varepsilon{\varepsilon}<\varepsilon$. By using the triangle inequality, the results of 4.II, and Lemma~\ref{lemma_x} with
$m=m_1(1+m_2)$, we obtain
$$
\begin{aligned}
  &\|x(t){-}x^*\|=\|x(t){-}x^*{-}x(t_{in}^\varepsilon\varepsilon){+}x(t_{in}^\varepsilon\varepsilon)\|\\
  &\le\|x(t_{in}^\varepsilon\varepsilon){-}x^*\|{+}\|x(t){-}x(t_{in}^\varepsilon\varepsilon)\|\\
 & \le \|x(t_{in}^\varepsilon\varepsilon){-}x^*\|+\frac{ M}{ L}\|x(t_{in}^\varepsilon\varepsilon){-}x^*\|^{m_1(1+m_2)}(e^{\nu L\varepsilon}{-}1)\\
 &\le\sqrt[2m_1]{\tfrac{\gamma_2}{\gamma_1}}\|x^0-x^*\|\varphi_{\tilde m}({\lambda t_{in}^\varepsilon\varepsilon}) \Big(1+\frac{ M}{ L}\Big(\tfrac{\gamma_2}{\gamma_1}\Big)^{\tilde m/2}(\|x^0-x^*\|\varphi_{\tilde m}({\lambda t_{in}^\varepsilon\varepsilon}))^{m_1\tilde m}(e^{\nu L\varepsilon}{-}1)\Big)\\
 &:=\sigma(t)\sqrt[2m_1]{\tfrac{\gamma_2}{\gamma_1}}\|x^0-x^*\|\varphi_{\tilde m}(\lambda{t_{in}^\varepsilon\varepsilon}),
 \end{aligned}
$$
This proves the second assertion of Theorem~3.
 \begin{center}
  \textit{{Proof of practical exponential and asymptotic stability}}
   \end{center}
\textit{Step 3.I.}  Assume that A4 is not satisfied,  $F_{0i}(J(x))>\alpha$ in $D$. We fix a $\rho\in(0,\delta)$, and suppose that $x^0\in B_{\delta}(x^*)\setminus B_{\rho}(x^*)\subset \mathcal L_{c_0}$.   The case $x^0\in\overline{B_{\rho}(x^*)}$ will be
covered on Step 4.I.
For  any $\rho_0{\in}(0,\rho)$, $\rho_{\min}{\in}(0,\rho_0)$, let $\tilde d={\min\{\rho-\rho_0,\rho_0-\rho_{\min}\}}$, and
$
0<\tilde\varepsilon_{0}<\Big({2\sqrt{2\pi} L\sum\limits_{i=1}^n\sqrt{k_i}}\Big)^{-2}\ln^2\Big(\tfrac{L \tilde d}{M}+1\Big).
$
Then:\\
P1 $x^0\in B_{\rho_0}(x^*)\Rightarrow x(t)\in B_{\rho}(x^*)\text{ for }t\in[0,\varepsilon];$\\
P2 $x^0\in B_{\rho}(x^*)\setminus B_{\rho_0}(x^*)\Rightarrow \|x(t)-x^*\|>0\text{ for }t\in[0,\varepsilon].$\\
From~P2, representation~\eqref{volt} holds for all $x^0 {\in} B_{\delta}(x^*){\setminus} B_{\rho_0}(x^*).$ 
\textit{Step 4.I.}
Applying Lemma~\ref{lemma_v} with $m_2{=}0$, Lemma~\ref{lemma_r} with $$m{=}0, \varpi{=}0, \tilde H=\max_{\underset{i,j,l=\overline{1,n}}{x\in D_0; s,p,q=1,2;}}\|L_{F_{ql}}L_{F_{pj}}F_{si}(J(x))\|,$$ and taking into account
$J(x^0)\ge\gamma_1\rho_0^{2m_1}>0$, we conclude that, for all $\varepsilon\in(0,\min\{\tilde\varepsilon_0,\varepsilon_1,\varepsilon_2\}]$
\begin{align*}
&\|R(\varepsilon)\|\le \varepsilon^{3/2}H\big(2\sqrt{2\pi}\sum_{i=1}^n\sqrt{k_i}\big)^{3}:=\varepsilon^{3/2}\tilde\Omega,\\
&J(x(\varepsilon))\le J(x^0)\Big(1- \tfrac{\varepsilon J^{\tilde m}(x^0)}{m_1}\big(\alpha\kappa_1-\sqrt{\varepsilon}c_0^{\tilde m/2}\tilde\lambda_1\big)\Big)^{m_1},\\
&\tilde\lambda_1={\sqrt{\kappa_2}\tilde\Omega}{\gamma_1^{\tfrac{1}{2m_1}-1}\rho_0^{1-2m_1}}+((m_1-1)\kappa_2+\mu m_1)\Big(\alpha\sqrt{\kappa_2}+{\sqrt\varepsilon\tilde\Omega}{\gamma_1^{\tfrac{1}{2m_1}-1}\rho_0^{1-2m_1}}\Big)^2.
\end{align*}
Similarly to Step 4.II, we may define an $\bar\varepsilon>0$ ensuring the following property: for any $\varepsilon\in(0,\bar\varepsilon]$, $\lambda\in(0,\bar\lambda]$, the property~$\tfrac{\varepsilon J(x^0)\lambda}{m_1}<1$ holds, and  the  solutions of
system~\eqref{int_n} with the controls $u_{is}^\varepsilon(t)$ and initial conditions from $B_{\delta}(x^*)\setminus B_{\rho}(x^*)$ satisfy the  inequality~\eqref{Jest1}.
\\
\textit{Step 5.I.} Following the argumentation from Step 5.II,  we conclude that there is an $N\ge0$ such that  $\|x(n\varepsilon)-x^*\|>\rho_0$ for $n=0,1,\dots,N-1$, $\|x(N\varepsilon)-x^*\|\le\rho_0$. Namely, {$N\ge
-\tfrac{1}{\bar\lambda\bar\varepsilon}\ln\tfrac{\rho_0}{\sigma\delta}$} if $m_1=1$, and {$N\ge 1-  \tfrac{m_1}{\bar\lambda\sigma_2\bar\varepsilon}(\sigma_1\delta^{-2m_1\tilde m}-\rho_0^{-2m_1\tilde m})$} if $m_1>1$.
Since $\|x(n\varepsilon)-x^*\|>\rho_0$ for $n=0,1,\dots,N-1$, representation~\eqref{volt} remains valid for all $n=0,1,\dots,N-1$ because of~P2.
Furthermore, similar to 5.II, we apply Lemma~\ref{lemma_x} with $m=0$, $\varepsilon\in(0,\min\{\bar \varepsilon, \tilde\varepsilon_0\})$, for $t<N\varepsilon$:
\begin{align}\label{practest}
\|x(t){-}x^*\|&\le  \sqrt[2m_1]{\tfrac{\gamma_2}{\gamma_1}}\|x^0-x^*\|\varphi_{\tilde m}({\lambda t_{in}^\varepsilon\varepsilon}) +\rho-\rho_0\\
  &\le  \sqrt[2m_1]{\tfrac{\gamma_2}{\gamma_1}}\|x^0-x^*\|\varphi_{\tilde m}({\lambda (t-\varepsilon)}) +\rho-\rho_0,\nonumber
 \end{align}
 where $\varphi_{\tilde m}(\lambda t)$ is given by~\eqref{varphi} with $\sigma=1$.
  From~P1, $x(t)\in B_{\rho}(x^*)$ for all $t\in[N\varepsilon,(N+1)\varepsilon]$. Thus, there are two cases:\\
   i) if $x((N+1)\varepsilon)\in B_{\delta}(x^*)\setminus B_{\rho_0}(x^*)$, we conclude  that~\eqref{practest} holds for all $t\le (N+1)\varepsilon$;\\
   ii) if $x((N+1)\varepsilon)\in  B_{\rho_0}(x^*)$, the property~P1 yields $x(t)\in B_{\rho}(x^*)$ for all $t\in[(N+1)\varepsilon,(N+2)\varepsilon]$. \\
    Repeating i),ii), we prove that $x(t)\in B_\rho(x^*)$ for all $t\ge N\varepsilon$. Recall that
 $\varphi_{\tilde m}({t_{in}^\varepsilon\varepsilon})$ is a positive decreasing function, and $\varphi_{\tilde m}({t_{in}^\varepsilon\varepsilon})\le \rho_0$, for all $t\ge (N+1)\varepsilon$.
 \\
 {Combining this with~\eqref{practest}, we complete the proof of Theorem~3.}
\hspace{18em}  $\square$
\end{appendix}

\end{document}